\documentclass[11pt]{article}
\usepackage[utf8]{inputenc}
\usepackage{color}
\usepackage{xcolor}
\usepackage{ulem}
\usepackage{lipsum}
\usepackage{authblk}
\usepackage{amsmath,amsfonts,amssymb,amsthm,times} 
\usepackage{hyperref}
\usepackage[english]{babel}
\usepackage[T1]{fontenc} 
\usepackage[utf8]{inputenc}
\usepackage{amsthm}
\usepackage{bbm}
\usepackage{amsmath}
\usepackage{amssymb}  
\usepackage{indentfirst}
\usepackage{fancyhdr}
\usepackage{amsthm}
\usepackage{graphicx}
\usepackage{pdfpages}
\usepackage{esint}
\usepackage{url}
\usepackage{xcolor}
\allowdisplaybreaks

\def\XXint#1#2#3{{\setbox0=\hbox{$#1{#2#3}{\int}$}
\vcenter{\hbox{$#2#3$}}\kern-.5\wd0}}

\makeatletter 
\newbox\tr@tto
\setbox\tr@tto=\hbox{{\count0=0\dimen0=-,9pt\dimen1=1,1pt\loop\ifnum
    \count0<11 \advance \count0 by1 \vrule width.51pt height\dimen1
    depth\dimen0\kern-0.17pt\advance\dimen0 by-0.05pt\advance\dimen1
    by0.1pt\repeat \loop\ifnum\count0<21\advance \count0 by1 \vrule
    width.6pt height\dimen1 depth\dimen0\kern-0.2pt \advance\dimen0
    by-0.1pt\advance\dimen1 by 0.05pt\repeat}}
\def\medint{\displaystyle\copy\tr@tto\kern-10.4pt\int}
\makeatother

\theoremstyle{definition}

\newtheorem{thm}{Theorem}[section] 
\newtheorem{cor}[thm]{Corollary} 
\newtheorem{lem}[thm]{Lemma} 
\newtheorem{prop}[thm]{Proposition} 
 
\newtheorem{dfn}[thm]{Definition} 

\numberwithin{equation}{section}

\def\proofof#1{\begin{proof}[Proof of #1]}

\definecolor{custom-blue}{RGB}{0,99,166} 
\hypersetup{colorlinks=true, allcolors=custom-blue}

\begin{document}
\author {Miriam Piccirillo 
\sc{}\thanks{Dipartimento di Matematica e Applicazioni "R. Caccioppoli", Università degli Studi di Napoli "Federico II", Via Cintia, 80126 Napoli, Italy. E-mail: \textit{miriam.piccirillo@unina.it}} 
  }

\title{Regularity for a strongly degenerate equation with explicit $u$-dependence}
\date{}
\maketitle

\begin{abstract}
We consider local weak solutions of widely degenerate elliptic PDEs of the type 
\begin{equation}
        \label{equazione mia}
        \mathrm{div}\Biggl(a(x)(|Du|-1)^{p-1}_+\frac{Du}{|Du|}\Biggr)=b(x,u) \ \ \text{ in }\Omega,
    \end{equation}
where $2\leq p<n,\textbf{ } \Omega$ is an open subset of $\mathbb{R}^n,n>2,$ and $( \ \cdot \ )_+$  stands for the positive part. We establish a higher differentiability result for the composition of the gradient with a suitable function that vanishes in the unit ball for the gradient, under suitable assumptions on the datum $b(x,u)$ and the coefficient $a(x).$ The novelty here with respect to previous papers on the subject is that the right hand side explicitly depends on the solution $u.$
\end{abstract}


\section{Introduction}
In this paper, we are interested in the regularity properties of the local weak solutions to strongly degenerate elliptic equation of the form
\begin{equation}
        \label{equazione mia}
        \mathrm{div}\Biggl(a(x)(|Du|-1)^{p-1}_+\frac{Du}{|Du|}\Biggr)=b(x,u) \ \ \text{ in }\Omega,
    \end{equation}
where $2\leq p<n$, $\Omega$ is an open subset of $\mathbb{R}^n,n>2,$ and $(\ \cdot \ )_+$  stands for the positive part. When dealing with equations as in \eqref{equazione mia}, one has to face two different major difficulties: the principal part of the equation is widely degenerate and the right hand side explicitly depends on the solution itself. Actually, the operator on the left hand side of \eqref{equazione mia} is uniformly elliptic only outside the unit ball centered at the origin, where it behaves asymptotically as the classical $p-$Laplace operator. One of the main motivations for the study of equations of this type comes from the optimal transport problems with congestion effects, and when $a(x)=1$ and $b(x,u)=f(x)$ such connection has been exhaustively studied in \cite{BraCa2, BraCa, BraCaSan, BS}.\\
The study of such an equation fits into the wider class of the asymptotically regular problems that have been extensively studied starting from the pioneering paper by Chipot and Evans \cite{CE}, concerning the homogeneous, autonomous quadratic growth case. Later on, still for the homogeneous, autonomous case, the Lipschitz continuity of weak solutions has been established for the superquadratic growth \cite{GiaMo}, and in the subquadratic growth case \cite{LPdNV}. Since then, many contributions to the regularity theory of weak solutions of widely degenerate equations have been established. Among the others we quote the results in \cite{BoDuGiPa, BoDuGiPa1, Brasco, C, CGPdN, CoFi2, EMM, FPdNV2, Gr, Ru}. It is worth to point out that no more than Lipschitz regularity can be expected for the solutions of widely degenerate problems. Indeed, every $1$-Lipschitz function is a solution of the homogeneous elliptic equation 
\begin{equation*}
\mathrm{div}\Biggl((|Du|-1)^{p-1}_+\frac{Du}{|Du|}\Biggr)=0.
\end{equation*}
However, many higher differentiability results are available for a suitable function of the gradient that vanishes outside the unit ball. This is a common aspect of nonlinear elliptic problems with growth rate $p \geq 2$ for which the higher differentiability is proven for a nonlinear function of the gradient that takes into account the growth and the structure of the equation.\\ Indeed, already for the $p$-Poisson equation, the higher differentiability is established for the function
\begin{equation}\label{vp}
 V_p(Du)=|Du|^{\frac{p-2}{2}}Du    
\end{equation}

as can be seen already in the pioneering paper by Uhlenbeck \cite{Uhlenbeck}. In case of widely degenerate problems, this phenomenon persists and higher differentiability results hold true for the function 
$$H_{p/2}(Du)=(|Du|-1)^{p/2}_+\frac{Du}{|Du|}$$
(see \cite{Am2,BraCa, BraCaSan, CGPdN2, CGM}), that, as one may expect, does not provide any information about the second-order regularity of the solutions in the set where the equation becomes degenerate.\\
Differently, from the above mentioned papers our datum depends on the solution itself. Even in the simplest setting of the Laplacian, the dependence on $u$ constitutes a difficulty by itself and this has been clear since the pioneering paper by Brezis and Nirenberg \cite{BreNir}, where the aim was to prove the existence of a positive solution to the elliptic equation 
\begin{equation*}
-\triangle u=u^p+f(x,u)  \ \ \ \text{on } \Omega.
\end{equation*}
Already in \cite{BreNir}, a critical exponent was identified, for the growth of the function $f(x,u)$ with respect to the function $u$. Later on, the problem has been studied in many different directions, replacing the Laplacian with more general non-linear operator and trying to find suitable conditions on the right hand side, to establish existence, regularity and the symmetry of solutions, (see for example \cite{Ba, Bi, DaP, DaS, ES,GuJ, Mon, MP, Sciu1, Sciu2}).\\
In this work we combine both the difficulties mentioned above with the aim of investigating the higher differentiability of solutions. Actually, the $W^{2,2}$ regularity for solutions of equations whose right hand side depends on the solution itself has been widely investigated, since, under suitable sign assumption on the right hand side yields a comparison principle for the solutions. The first paper that deals with widely degenerate equation with right hand side depending on the solution is \cite{GRu}, where the authors were able to prove a higher differentiability result for $H_{p/2}(Du)$, under quite more restrictive conditions on the right hand side. In fact, with the respect to the result in \cite{GRu}, not only we deal with weaker assumptions on $b(x,u)$ but also allow for discontinuous coefficients, assuming on the function $a(x)$ only a $W^{1,n}$ regularity. \\To state and describe our results precisely let us specify the assumptions on our data.\\
The coefficient $a(x)$ is assumed to satisfy 
\begin{align}
\label{ipoa}\nonumber
&{\mathrm{(}\mathrm{i}\mathrm{)}} \textbf{ }a(x)\in W^{1,n}_{loc}(\Omega),\\
&{\mathrm{(}\mathrm{ii}\mathrm{)}}\textbf{ }\exists\textbf{ } m,M>0: m<a(x)<M \ \ \ \text{a.e. in }\Omega.
\end{align}
For the datum $b(x,u)$ we assume that there exists an exponent $q$ such that
\begin{equation}
\label{q}
   1\leq q\leq p^*\bigg(\frac{p-2}{p}\bigg)+1,
\end{equation}
and non-negative functions $k(x)$ and $g(x)$ such that 
\begin{align}
\label{ipob1}\nonumber
    &{\mathrm{(}\mathrm{i}\mathrm{)}} \textbf{ }|b(x,z)-b(y,z)|\leq(k(x)+k(y))|x-y||z|^q,\\
    &{\mathrm{(}\mathrm{ii}\mathrm{)}} \textbf{ }|b(x,z)-b(x,t)|\leq g(x)|z-t|(|z|^{q-1}+|t|^{q-1}),
\end{align}
for a.e. $x,y\in \Omega$ and every $s,t\in \mathbb{R}$.
As usual, $p^*$ denotes the Sobolev conjugate exponent of $p$, i.e. \begin{equation*}
p^*=\frac{np}{n-p}\end{equation*}
It is worth noting that the case $p \geq n$ is less interesting from a regularity perspective. Indeed, by the Sobolev embedding theorem, any local weak solution $u \in W_{loc}^{1,p}(\Omega)$ is automatically locally bounded and Hölder continuous. Consequently, the upper bound for the exponent $q$ in \eqref{q} becomes trivial, as $p^*$ can be chosen arbitrarily large. Under these conditions, the right-hand side $b(x,u)$ is locally controlled, by a function of the only variable $x$ and hence the standard regularity theory for widely degenerate equations already guarantees both the higher differentiability and the local Lipschitz continuity of the solutions.
 For the functions appearing in \eqref{ipob1}, we shall assume that
\begin{align}
\label{s e gam 1}
&k(x)\in L^s_{loc}(\Omega) \ \ \ \text{ where } \ \ \frac{1}{s}=1-\frac{q}{p^*}-\frac{n-2}{np},\\
\label{s e gam 2}
& g(x)\in L^\gamma_{loc}(\Omega)\ \ \ \text{ where } \ \ \ \frac{1}{\gamma}=1-\frac{q-1}{p^*}-\frac{2(n-1)}{np}.
\end{align}
Note that by \eqref{q} both $s$ and $\gamma$ are greater than $1.$
As it is natural, the values of $s$ and $\gamma$ depend on $q,$ more precisely when $q$ goes to $1$ we obtain a datum $b(x,u)=b(x)$ that does not depend on the solution $u$, while when $q$ goes to its critical value we obtain specific ranges for the values of $s$ and $\gamma$, 
\begin{equation}
 \begin{cases}
\frac{np}{np-2n+p+2}\leq s_q\leq \frac{np}{p+2},\\
\frac{np}{np-2n+2}\leq \gamma_q\leq \frac{np}{2}.
 \end{cases}   
\end{equation}
As an example of the datum $b(x,u)$ that satisfies \eqref{ipob1} we can consider $$b(x,u)=f(x)z(u) \text{ where }f(x)\in W^{1,s}(\Omega)\text{ and }|z(u)|\leq C|u|^q.$$
Moreover, for simplicity, we shall assume that $b(x,0)=0$ which yields, through the second assumption in \eqref{ipob1}, that \begin{equation}
    \label{ipb}
    |b(x,u)|\leq g(x)|u|^q.  
\end{equation}
Our main result is the following 
\begin{thm}
\label{teo3}
   Let $u\in W^{1,p}_{loc}(\Omega)$ be a weak solution of the equation \eqref{equazione mia}, under assumptions \eqref{ipoa}, \eqref{q}, \eqref{ipob1},  \eqref{s e gam 1} and \eqref{s e gam 2}. Then $H_{p/2}(Du)\in W^{1,2}_{loc}(\Omega)$ and the following estimate holds 
\begin{align*}
&\int_{B_{R/2}}|D(H_{p/2}(Du))|^2dx\leq C\int_{B_R}(1+|Du|^p)dx\\
&+C\bigg(\int_{B_{R}}|u|^{p}+|Du|^pdx\bigg)^{\frac{q}{p-1}}\bigg[\bigg(\int_{B_{R_0}}|k(x)|^sdx\bigg)^{\frac{p}{s(p-1)}}+\bigg(\int_{B_{R_0}}|g(x)|^\gamma dx\bigg)^{\frac{p}{\gamma(p-1)}}\bigg],   
\end{align*}
for every $B_R\subset B_{R_0}\Subset\Omega$ and a constant $C=C(n, p, q, s,\gamma, m, M, R).$
\end{thm}
Let us briefly summarize the proof of the previous Theorem, describing its main points. First of all, we freeze and regularize the datum $b(x,u)$ and regularize the coefficient $a(x)$, in order to cast the problem into a classical variational framework. This way of arguing leads us to the study of the higher differentiability of the solutions to the problem \eqref{equazione nuova} below, in which the right hand side is independent of the solution.\\ The central part of the proof consists in establishing a suitable \textit{a priori} estimate for the derivatives of the composition between the function $H_{p/2}(\xi)$ and the gradient of the solution of the frozen problem. Next, we introduce a family of regularized problems whose solutions satisfy the assumptions of the a priori estimate and we show that this estimate is preserved in passing to the limit. Finally, we come back to the original problem proving that the limit function coincides with the solution of \eqref{equazione mia} outside the ball of radius $1$. This allows us to conclude that the solution of \eqref{equazione mia} shares with this limit function the same higher differentiability property.

\section{Preliminaries \label{sec:prelim}}
\selectlanguage{british}%

\subsection{Notation and essential definitions }

\selectlanguage{english}%
\noindent $\hspace*{1em}$In this paper we shall denote by $C$ or
$c$ a general positive constant that may vary on different occasions.
Relevant dependencies on parameters and special constants will be
suitably emphasized using parentheses or subscripts. \foreignlanguage{british}{The
norm we use on $\mathbb{R}^{k}$, }\foreignlanguage{american}{$k\in\mathbb{N}$}\foreignlanguage{british}{,
will be the standard Euclidean one and it will be denoted by $\left|\,\cdot\,\right|$.
In particular, for the vectors $\xi,\eta\in\mathbb{R}^{k}$, we write
$\langle\xi,\eta\rangle$ for the usual inner product and $\left|\xi\right|:=\langle\xi,\xi\rangle^{\frac{1}{2}}$
for the corresponding Euclidean norm.}\\
$\hspace*{1em}$In what follows, $B_{r}(x_{0})=\left\{ x\in\mathbb{R}^{n}:\left|x-x_{0}\right|<r\right\} $
will denote the $n$-dimensional open ball centered at $x_{0}$ with
radius $r$. We shall sometimes omit the dependence on the center
when all balls occurring in a proof are concentric. Unless otherwise
stated, different balls in the same context will have the same center.\\
\foreignlanguage{british}{$\hspace*{1em}$For further needs, we now
define the auxiliary function $H_{\delta}:\mathbb{R}^{n}\rightarrow\mathbb{R}^{n}$
by 
\begin{equation}
H_{\delta}(\xi):=\begin{cases}
\begin{array}{cc}
(\vert\xi\vert-1)_{+}^{\delta}\,\frac{\xi}{\left|\xi\right|} & \,\,\mathrm{if}\,\,\,\xi\neq0,\\
0 & \,\,\mathrm{if}\,\,\,\xi=0,
\end{array}\end{cases}\label{eq:Hfun}
\end{equation}
}

\selectlanguage{british}%
\noindent where $\delta>0$ is a parameter. We
conclude this first part of the preliminaries by recalling the following
definition.
\begin{dfn}
\noindent A function $u\in W_{loc}^{1,p}(\Omega)$
is a \textit{local weak solution} of equation (\ref{equazione mia})
if and only if, for any test function $\varphi\in W_{0}^{1,p}$$(\Omega')$ with $\Omega'\Subset\Omega$,
the following integral identity holds:
\[
-\int_{\Omega}\langle a(x)H_{p-1}(Du),D\varphi\rangle\,dx\,=\,\int_{\Omega}b(x,u)\varphi\,dx.
\]
\end{dfn}It is worth noting that the integral identity above is well-posed. Indeed, the finiteness of the left-hand side easily follows from Hölder's inequality, the boundedness of $a(x)$ assumed in \eqref{ipoa}, and the fact that $|H_{p-1}(Du)| \leq |Du|^{p-1} \in L_{loc}^{\frac{p}{p-1}}(\Omega)$. On the other hand, the integrability of the right-hand side is guaranteed by the growth condition \eqref{ipb}, the integrability assumption \eqref{s e gam 2} on $g(x)$, and the Sobolev embedding theorem.

\subsection{Algebraic inequalities }

\noindent $\hspace*{1em}$In this section, we gather some relevant
algebraic inequalities that will be needed later on.
\noindent We recall the following estimate, whose proof can be
found in \cite[Chapter 12]{Lind}.
\begin{lem}
\noindent \label{lem:Lind} Let $p\in[2,\infty)$ and $k\in\mathbb{N}$.
Then, for every $\xi,\eta\in\mathbb{R}^{k}$, the following inequality
\[
\vert\xi-\eta\vert^{p}\leq\,C\left|\vert\xi\vert^{\frac{p-2}{2}}\xi-\vert\eta\vert^{\frac{p-2}{2}}\eta\right|^{2}
\]
holds for a constant $C\equiv C(p)>0$.\\
Moreover, let \foreignlanguage{british}{$p\in(1,\infty)$}. Then, there exists a constant $c\equiv c(n,p)>0$
such that 
\begin{equation}
\langle |\xi|^{p-2}\xi-|\eta|^{p-2}\eta,\xi-\eta\rangle\,\geq\,c\bigg||\xi|^{\frac{p-2}{2}}\xi-|\eta|^{\frac{p-2}{2}}\eta\bigg|^2,
\end{equation}
for every $\xi,\eta\in\mathbb{R}^{n}$.
\end{lem}

\selectlanguage{english}%
\noindent Combining \cite[Lemma 2.2]{AceFu} with \cite[Formula (2.4)]{GiaMo},
we obtain the following
\begin{lem}
\label{D1} Let $1<p<\infty$. There exists a constant $c\equiv c(n,p)>0$
such that 
\begin{center}
$c^{-1}(|\xi|^{2}+|\eta|^{2})^{\frac{p-2}{2}}\leq\dfrac{\left|\vert\xi\vert^{\frac{p-2}{2}}\xi-\vert\eta\vert^{\frac{p-2}{2}}\eta\right|^{2}}{|\xi-\eta|^{2}}\leq c\,(|\xi|^{2}+|\eta|^{2})^{\frac{p-2}{2}}$ 
\par\end{center}
\noindent for every $\xi,\eta\in\mathbb{R}^{n}$ with $\xi\neq\eta$. 
\end{lem}

\selectlanguage{british}%
\noindent $\hspace*{1em}$\foreignlanguage{english}{For the function
$H_{p-1}(\xi)$ defined at (\ref{eq:Hfun}) with $\delta=p-1$, we record
the following estimates, } which are precisely \foreignlanguage{english}{\cite[Lemma 2.5]{AmG} and \cite[Lemma 4.1]{BraCaSan}
with $\lambda=1$ }.

\begin{lem}
\label{lem:Brasco} Let \foreignlanguage{british}{$p\in(1,\infty)$}. Then, there exists a constant $c\equiv c(n,p)>0$
such that 
\begin{equation}
\langle H_{p-1}(\xi)-H_{p-1}(\eta),\xi-\eta\rangle\,\geq\,c\,\vert H_{\frac{p}{2}}(\xi)-H_{\frac{p}{2}}(\eta)\vert^{2},\label{eq:BraAmb}
\end{equation}
for every $\xi,\eta\in\mathbb{R}^{n}$. \foreignlanguage{british}{Moreover,
if $2\leq p<\infty$ we have 
\[
\vert H_{p-1}(\xi)-H_{p-1}(\eta)\vert\leq (p-1)\bigg(\vert H_{p/2}(\xi)|^{\frac{p-2}{p}}+|H_{p/2}(\eta)\vert^{\frac{p-2}{p}}\bigg)\vert H_{p/2}(\xi)-H_{p/2}(\eta)\vert.
\]
}
\end{lem}
For further needs we recall also the following result, whose can be found in \cite[Lemma 2.3]{AM}
\begin{lem}
\label{eq:BraAmb}
For every $(\alpha,\varepsilon)\in \mathbb{R}^+\times\mathbb{R}^+$ with $\alpha<\varepsilon,$ \foreignlanguage{british} there exist two positive constants $\beta_1(\alpha,\varepsilon)$ and $\beta_2(\alpha,\varepsilon, n)$
such that 
\begin{equation}
\beta_1 |H_{\varepsilon}(\xi)-H_{\varepsilon}(\eta)|\leq\,\,\frac{\vert H_{\alpha}(\xi)-H_{\alpha}(\eta)\vert}{\big((|\xi|-1)^\varepsilon_++(|\eta|-1)^\varepsilon_+\big)^{\frac{\alpha-\varepsilon}{\varepsilon}}}\leq\beta_2|H_{\varepsilon}(\xi)-H_{\varepsilon}(\eta)|
\end{equation}
for every $\xi,\eta\in\mathbb{R}^{n}$. 
\end{lem}

We conclude by recalling a well-known iteration Lemma, whose proof can be found in \cite[Lemma 6.1]{Giu}.
\begin{lem}
    \label{lem:Giusti2} Let \foreignlanguage{british}{$Z(t)$} be a bounded non-negative function in the interval $[\rho,R].$ Assume that for $\rho\leq r<t\leq R$ we have $$Z(r)\leq[A(t-r)^{-\alpha}+B(t-r)^{-\beta}+C]+\theta Z(t)$$ with $A, B, C\geq0,\textbf{ } \alpha>\beta>0$ and $0\leq \theta<1.$ Then,
    $$Z(\rho)\leq c(\alpha,\theta)[A(R-\rho)^{-\alpha}+B(R-\rho)^{-\beta}+C].$$
\end{lem}

\subsection{Difference quotients}

\label{subsec:DiffOpe}

\noindent $\hspace*{1em}$We recall here the definition and some elementary
properties of the difference quotients that will be useful in the
following (see, for example, \cite{Giu}). 
\begin{dfn}
\noindent For every vector-valued function $F:\mathbb{R}^{n}\rightarrow\mathbb{R}^{k}$
the \textit{finite difference operator }in the direction $x_{j}$
is defined by 
\[
\tau_{j,h}F(x)=F(x+he_{j})-F(x),
\]
where $h\in\mathbb{R}$, $e_{j}$ is the unit vector in the direction
$x_{j}$ and $j\in\{1,\ldots,n\}$.\\
 $\hspace*{1em}$The \textit{difference quotient} of $F$ with respect
to $x_{j}$ is defined for $h\in\mathbb{R}\setminus\{0\}$ by 
\[
\Delta_{j,h}F(x)\,=\,\frac{\tau_{j,h}F(x)}{h}\,.
\]
\end{dfn}

\noindent When no confusion can arise, we shall omit the index $j$
and simply write $\tau_{h}$ or $\Delta_{h}$ instead of $\tau_{j,h}$
or $\Delta_{j,h}$, respectively. 
\selectlanguage{british}%
\begin{prop}
\label{prop}
\noindent Let $\Omega\subset\mathbb{R}^{n}$ be an open set and let
$F\in W^{1,q}(\Omega)$, with $q\geq1$. Moreover, let $G:\Omega\rightarrow\mathbb{R}$
be a measurable function and consider the set
\[
\Omega_{\vert h\vert}:=\left\{ x\in\Omega:\mathrm{dist}\left(x,\partial\Omega\right)>\vert h\vert\right\} .
\]
\foreignlanguage{english}{Then:}\\
\foreignlanguage{english}{}\\
\foreignlanguage{english}{$\mathrm{(}\mathrm{i}\mathrm{)}$ $\Delta_{h}F\in W^{1,q}\left(\Omega_{\vert h\vert}\right)$
and $\partial_{i}(\Delta_{h}F)=\Delta_{h}(\partial_{i}F)$ for every
$\,i\in\{1,\ldots,n\}$.}\\

\selectlanguage{english}%
\noindent $\mathrm{(}\mathrm{ii}\mathrm{)}$ If at least one of the
functions $F$ or $G$ has support contained in $\Omega_{\vert h\vert}$,
then 
\[
\int_{\Omega}F\,\Delta_{h}G\,dx\,=\,-\int_{\Omega}G\,\Delta_{-h}F\,dx.
\]
$\mathrm{(}\mathrm{iii}\mathrm{)}$ We have 
\[
\Delta_{h}(FG)(x)=F(x+he_{j})\Delta_{h}G(x)\,+\,G(x)\Delta_{h}F(x).
\]
\end{prop}

\selectlanguage{english}%
\noindent The next result about the finite difference operator is
a kind of integral version of the Lagrange Theorem and its proof can
be found in \cite[Lemma 8.1]{Giu}. 
\begin{lem}
\noindent \label{lem:Giusti1} If $0<\rho<R$, $\vert h\vert<\frac{R-\rho}{2}$,
$1<q<+\infty$ and $F\in L^{q}(B_{R},\mathbb{R}^{k})$ is such that
$DF\in L^{q}(B_{R},\mathbb{R}^{k\times n})$, then 
\[
\int_{B_{\rho}}\left|\tau_{h}F(x)\right|^{q}dx\,\leq\,c^{q}(n)\,\vert h\vert^{q}\int_{B_{R}}\left|DF(x)\right|^{q}dx.
\]
Moreover 
\[
\int_{B_{\rho}}\left|F(x+he_{j})\right|^{q}dx\,\leq\,\int_{B_{R}}\left|F(x)\right|^{q}dx.
\]
\end{lem}

\noindent Finally, we recall the following fundamental result, whose
proof can be found in \cite[Lemma 8.2]{Giu}. 
\begin{lem}
\noindent \label{lem:RappIncre} Let $F:\mathbb{R}^{n}\rightarrow\mathbb{R}^{k}$,
$F\in L^{q}(B_{R},\mathbb{R}^{k})$ with $1<q<+\infty$. Suppose that
there exist $\rho\in(0,R)$ and a constant $M>0$ such that 
\[
\sum_{j=1}^{n}\int_{B_{\rho}}\left|\tau_{j,h}F(x)\right|^{q}dx\,\leq\,M^{q}\,\vert h\vert^{q}
\]
for every $h\in\mathbb{R}$ with $\vert h\vert<\frac{R-\rho}{2}$.
Then $F\in W^{1,q}(B_{\rho},\mathbb{R}^{k})$. Moreover 
\[
\Vert DF\Vert_{L^{q}(B_{\rho})}\leq M
\]
and 
\[
\Delta_{j,h}F\rightarrow\partial_{j}F\,\,\,\,\,\,\,\,\,\,in\,\,L_{loc}^{q}(B_{R},\mathbb{R}^{k})\,\,\,\,\mathit{as}\,\,h\rightarrow0,
\]
for each $j\in\{1,\ldots,n\}$. 
\end{lem}
For further needs, we record the following
\begin{lem}
\label{lem:sob}
    Let $\Omega\subset\mathbb{R}^n$ be a bounded open set, $p\geq 2$ and $u\in W^{1,p}_{loc}(\Omega).$ Then the implication $$H_{p/2}(Du)\in W^{1,2}_{loc}(\Omega)\Longrightarrow |Du|\in L^{\frac{np}{n-2}}_{loc}(\Omega)$$ holds true, together with the estimate
$$\bigg(\int_{B_R}|Du|^{\frac{pn}{n-2}}dx\bigg)^{\frac{n-2}{n}}\leq C\bigg(\int_{B_R}|DH_{p/2}(Du)|^2dx+\int_{B_R}(1+|Du|^p)dx\bigg).$$
for every $B_R\Subset\Omega$
\end{lem}
\begin{proof}
    By the Sobolev embedding theorem we have that $$H_{p/2}(Du)\in L^{2^*}_{loc}(\Omega)$$ with the estimate $$\bigg(\int_{B_R}|H_{p/2}(Du)|^{2^*}dx\bigg)^{\frac{1}{2^*}}\leq C\bigg(\int_{B_R}|DH_{p/2}(Du)|^2dx+\int_{B_R}|H_{p/2}(Du)|^2dx\bigg)^{\frac{1}{2}}.$$ By the defintion of $H_{p/2}(\xi)$ in \eqref{eq:Hfun}, previous estimate yields
    $$\bigg(\int_{B_R}(|Du|-1)_+^{\frac{pn}{n-2}}dx\bigg)^{\frac{n-2}{n}}\leq C\bigg(\int_{B_R}|DH_{p/2}(Du)|^2dx+\int_{B_R}(|Du|-1)^p_+dx\bigg),$$ and from this inequality the conclusion easily follows.
\end{proof}

\section{The regularized problem}
In this section we establish the higher differentiability of the solutions to an equation similar to that in \eqref{equazione mia} under higher regularity assumptions on the data $a(x)$ and $b(x,u)$ that will be instrumental for the proof of our main result. The main feature of this equation is that we freeze the right hand side $b(x,\cdot)$ calculating it on a fixed function belonging to $W^{1,p}_{loc}(\Omega)$. More precisely, let us fix a function $w\in W^{1,p}_{loc}(\Omega)$ and let $v\in W^{1,p}_{loc}(\Omega)$ be a solution of the following equation 
\begin{equation}
        \label{equazione nuova}
        \mathrm{div}\Biggl(a(x)(|Dv|-1)^{p-1}_+\frac{Dv}{|Dv|}\Biggr)=b(x,w) \text{ in }\Omega.
    \end{equation}
Our first aim is to prove the following
\begin{thm}
\label{teo1}
Assume that $a(x)$ satisfies $(ii)$ in \eqref{ipoa} and that for every ball $B_R\Subset\Omega$ there exists $L_{1,R}>0$ such that 
    \begin{equation}
        \label{ipoal}
\sup_{\substack{x,y\in B_R \\ x\not=y}}\frac{|a(x)-a(y)|}{|x-y|}\leq L_{1,R}.
    \end{equation}
Assume further that for every ball $B_R\Subset\Omega$ there exist constants $L_{2,R},L_{3,R}>0$ such that 
    \begin{align}
       \label{ipobforte1}
|b(x,z)-b(y,z)|&\leq L_{2,R}|x-y||z|^q,\\
\label{ipobforte2}
|b(x,z)-b(x,t)|&\leq L_{3,R}|z-t|(|z|^{q-1}+|t|^{q-1}), 
\end{align}
for every $x,y\in B_R$, every $z,t\in \mathbb{R}$ and where $q$ satisfies \eqref{q}. If $v\in W^{1,p}_{loc}(\Omega)$ solves \eqref{equazione nuova}, then $H_{p/2}(Dv)\in W^{1,2}_{loc}(\Omega)$ with the estimate
\begin{align*}
&\int_{B_\rho}|D(H_{p/2}(Dv))|^2dx\leq C\int_{B_R}(1+|Dv|^p)dx+C\bigg(\int_{B_R}|w|^p+|Dw|^pdx\bigg)^{\frac{q}{p-1}}.
\end{align*}
for all $B_\rho \subseteq B_{R/2}\subset B_R\Subset\Omega$ and $C=C(m, M, L_1, L_2, L_3, n, p, q,  \rho, R),$
\end{thm}
\noindent \begin{proof}[\bfseries{Proof}]
Since $v$ is a weak of solution of \eqref{equazione nuova}, then
\begin{equation}
   \label{int}
   -\int_\Omega \bigl\langle a(x)H_{p-1}(Dv),D\varphi\bigr\rangle \textbf{ }dx=\int_\Omega b(x,w)\varphi\textbf{ } dx \textbf{ }\text{ for all $\varphi \in W^{1,p}_0(\Omega)$}. 
\end{equation}
Fix a ball $B_R\Subset\Omega$ and let $\zeta \in C^\infty_0(B_{R/2})$ be a cut off function such that $0\leq\zeta\leq 1, \zeta=1$ in $B_{r/2}$ with $r< R$ and $|D\zeta|<\frac{C}{R-r}.$ In what follows, without loss of generality, we shall suppose that $R<1.$
Choosing $\varphi=\tau_{j,-h}(\zeta^2\tau_{j,h}v)$ as test function in \eqref{int}, with $h\in \mathbb{R}\setminus\{0\}$ and such that $|h|<R/4$, integrating by parts by means of $(ii)$ in Proposition \ref{prop}, we get
$$-\int_\Omega\biggl\langle\tau_{j,h}\bigg(a(x)H_{p-1}(Dv)\bigg),\zeta^2D(\tau_{j,h}v)+2\zeta D\zeta\tau_{j,h}v\biggr\rangle\textbf{ }dx=\int_\Omega \zeta^2 \tau_{j,h}(b(x,w))\tau_{j,h}v\textbf{ } dx.$$ 
By $(iii)$ in Proposition \ref{prop} and some elementary calculations, the previous equality can be written as follows
\begin{align*}
&0=\int_\Omega\zeta^2\bigl\langle\tau_{j,h}(a(x))H_{p-1}(Dv),\tau_{j,h}(Dv)\bigr\rangle\textbf{ }dx\\
&+\int_\Omega \zeta^2a(x+he_j)\bigl\langle\tau_{j,h}(H_{p-1}(Dv)),\tau_{j,h}(Dv)\bigr\rangle\textbf{ }dx\\
&+2\int_\Omega\zeta\bigl\langle\tau_{j,h}(a(x))H_{p-1}(Dv),\tau_{j,h}v D\zeta\bigr\rangle\textbf{ }dx\\
&+2\int_\Omega \zeta a(x+he_j)\bigl\langle\tau_{j,h}(H_{p-1}(Dv)),\tau_{j,h}v D\zeta\bigr\rangle\textbf{ }dx\\
&+\int_\Omega \zeta^2 \tau_{j,h}v[b(x+he_j,w(x+he_j))-b(x,w(x+he_j))]\textbf{ }dx\\
&+\int_\Omega \zeta^2 \tau_{j,h}v[b(x,w(x+he_j))-b(x,w(x))]\textbf{ }dx\\
& =:I_1+I_2+I_3+I_4+I_5+I_6,
\end{align*}
that yields 
\begin{equation}
\label{dis:genl}
    I_2\leq |I_1|+|I_3|+|I_4|+|I_5|+|I_6|.
\end{equation}  
The integral $I_2$ can be estimated using assumption $(ii)$ in \eqref{ipoa}  and the degenerate ellipticity of the equation given by the first estimate in Lemma \ref{lem:Brasco}, as follows
\begin{equation}
\label{i2l}
 I_2=\int_\Omega\zeta^2 a(x+he_j)\bigl\langle\tau_{j,h}(H_{p-1}(Dv)),\tau_{j,h}(Dv)\bigr\rangle dx\geq c_0(m,n,p)\int_\Omega\zeta^2|\tau_{j,h}(H_{p/2}(Dv))|^2dx.   
\end{equation}
By $(iii)$ in Proposition \ref{prop} , we may write 
\begin{align}
\label{prodl}
\nonumber
    H_{p-1}(Dv)\tau_{j,h}Dv&=\tau_{j,h}\Big(H_{p-1}(Dv)\cdot Dv\Big)-Dv(x+he_j)\tau_{j,h}\Big(H_{p-1}(Dv)\Big)\\ \nonumber
    &=\tau_{j,h}\Big((|Dv|-1)^{p-1}_+|Dv|\Big)-Dv(x+he_j)\tau_{j,h}\Big(H_{p-1}(Dv)\Big)\\ \nonumber
    &=\tau_{j,h}\Big((|Dv|-1)^{p}_+\Big)+\tau_{j,h}\Big((|Dv|-1)^{p-1}_+\Big)-Dv(x+he_j)\tau_{j,h}\Big(H_{p-1}(Dv)\Big)\\
    &=\tau_{j,h}\Big(|H_p(Dv)|\Big)+\tau_{j,h}\Big(|H_{p-1}(Dv)|\Big)-Dv(x+he_j)\tau_{j,h}\Big(H_{p-1}(Dv)\Big),
\end{align} where we used the definition of $H_\delta(\xi)$ both with $\delta=p$ and $\delta=p-1$ and the linearity of the finite difference operator.\\
For the estimate of $I_1$ we use \eqref{prodl}, thus getting

\begin{align}
\label{i1i2i3}
\nonumber
 I_1&\leq\int_\Omega\zeta^2|\tau_{j,h}(a(x))||\tau_{j,h}(|H_{p}(Dv)|)|dx+\int_\Omega\zeta^2|\tau_{j,h}(a(x))||\tau_{j,h}(|H_{p-1}(Dv)|)|dx\\\nonumber
&+\int_\Omega\zeta^2|\tau_{j,h}(a(x))|\bigg(|Dv(x+he_j)||\tau_{j,h}(H_{p-1}(Dv))|\bigg)dx\\
&=:|I_{1,1}|+|I_{1,2}|+|I_{1,3}|.
\end{align}
We now proceed estimating the integrals $|I_{1,j}|,\textbf{ }j=1,2,3$.\\
Using Lemma \ref{eq:BraAmb} with $\alpha=p/2$, $\varepsilon=p$, $\xi=Dv(x+he_j)$ and $\eta=Dv(x)$, we obtain
 \begin{equation}
 \label{3.8}
    |\tau_{j,h}(H_p(
    Dv))|\leq C\Big|\tau_{j,h}(H_{p/2}(Dv))\Big|\Big[(|Dv(x+he_j)|-1)^{p/2}_++(|Dv(x)|-1)^{p/2}_+\Big].
 \end{equation}
Using \eqref{3.8} and Young's inequality, we get

\begin{align}
\label{i11l}
\nonumber
|I_{1,1}|&\leq C\int_{\Omega}\zeta^2|\tau_{j,h}(a(x))|\Big|\tau_{j,h}H_{p/2}(Dv)\Big|\Big[(|Dv(x+he_j)|-1)^{p/2}_++(|Dv(x)|-1)^{p/2}_+\Big]dx, \\ \nonumber     
&\leq\varepsilon\int_\Omega\zeta^2\Big|\tau_{j,h}H_{p/2}(Dv)\Big|^2dx \\ \nonumber
&+C_\varepsilon\int_\Omega\zeta^2|\tau_{j,h}a(x)|^2\Big((|Dv(x+he_j)|-1)^p_++(|Dv(x)|-1)^p_+\Big) dx  \\ 
&\leq \varepsilon\int_\Omega\zeta^2\Big|\tau_{j,h}H_{p/2}(Dv)\Big|^2dx+C_\varepsilon(L_1)|h|^2\int_{B_R}(|Dv(x)|-1)^p_+dx,
\end{align}
where we used assumption  \eqref{ipoal}, the properties of $\zeta$ and the second estimate in Lemma \ref{lem:Giusti1}.
To estimate $|I_{1,2}|$, we use the second estimate in Lemma \ref{lem:Brasco}, assumption \eqref{ipoal} and the properties of $\zeta$ as follows
\begin{align}
\label{jaydefl}
\nonumber
|I_{1,2}|&\leq C\int_\Omega\zeta^2|\tau_{j,h}(a(x))|\bigg|\tau_{j,h}(H_{p/2}(Dv))\bigg|\bigg|\big|H_{p/2}(Dv(x+he_j))\big|^{\frac{p-2}{p}}+\big|H_{p/2}(Dv(x))\big|^{\frac{p-2}{p}}\bigg|dx\\ \nonumber
&\leq L_1|h|\int_\Omega\zeta^2\Big|\tau_{j,h}(H_{p/2}(Dv))\Big|\Big(\Big|H_{p/2}(Dv(x+he_j))\Big|^{\frac{p-2}{p}}+\Big|H_{p/2}(Dv(x))\Big|^{\frac{p-2}{p}}\Big)dx\\ \nonumber
&\leq \varepsilon\int_\Omega\zeta^2|\tau_{j,h}H_{p/2}(Dv)|^2dx +C_{\varepsilon}(L_1)|h|^2\int_{B_{R/2}}\Big(|H_{p/2}(Dv(x+he_j))|^{\frac{p-2}{p}}+|H_{p/2}(Dv(x))|^{\frac{p-2}{p}}\Big)^2dx\\ \nonumber
&\leq\varepsilon\int_\Omega\zeta^2|\tau_{j,h}H_{p/2}(Dv)|^2dx+C_\varepsilon(L_1)|h|^2\int_{B_{R}}|H_{p/2}(Dv)|^{\frac{2(p-2)}{p}}dx\\\nonumber
&\leq \varepsilon\int_\Omega\zeta^2|\tau_{j,h}H_{p/2}(Dv)|^2dx+C_\varepsilon(L_1)|h|^2\int_{B_{R}}(|Dv|-1)^{p-2}_+dx\\
&\leq \varepsilon\int_\Omega\zeta^2|\tau_{j,h}H_{p/2}(Dv)|^2dx+C_\varepsilon(L_1)|h|^2\int_{B_{R}}(1+|Dv(x)|)^pdx,
\end{align}
where we also used Young's inequality, the definition of $H_{p/2}(Dv)$ and the second estimate in Lemma \ref{lem:Giusti1}.\\
Using the second estimate in Lemma \ref{lem:Brasco}, Young's inequality, assumption \eqref{ipoal} and the properties of $\zeta$, we get
 
\begin{align}
\label{i12l} \nonumber
|I_{1,3}|&\leq(p-1)\int_\Omega\zeta^2|\tau_{j,h}a(x)||Dv(x+he_j)|\Big|\tau_{j,h}(H_{p/2}(Dv))\Big|\Big||H_{p/2}(Dv(x+he_j))|^{\frac{p-2}{p}}+|H_{p/2}(Dv)|^{\frac{p-2}{p}}\Big|dx\\ \nonumber
&\leq\varepsilon\int_\Omega \zeta^2|\tau_{j,h}H_{p/2}(Dv)|^2dx\\ \nonumber
&+C_\varepsilon\int_\Omega\zeta^2|\tau_{j,h}a(x)|^2|Dv(x+he_j)|^2\Big||H_{p/2}(Dv(x+he_j))|^{\frac{p-2}{p}}+|H_{p/2}(Dv)|^{\frac{p-2}{p}}\Big|^2dx\\ \nonumber
&\leq\varepsilon\int_\Omega\zeta^2|\tau_{j,h}H_{p/2}(Dv)|^2dx\\ \nonumber
&+C_\varepsilon(L_1)|h|^2\int_{B_{R/2}}|Dv(x+he_j)|^2\Big||H_{p/2}(Dv(x+he_j))|^{\frac{p-2}{p}}+|H_{p/2}(Dv)|^{\frac{p-2}{p}}\Big|^2dx\\
&\leq \varepsilon\int_\Omega \zeta^2|\tau_{j,h}H_{p/2}(Dv)|^2dx+C_\varepsilon(L_1)|h|^2\int_{B_R}|Dv|^pdx,
\end{align}
where in the last line we used the definition of $H_{p/2}(Dv)$ and second estimate in Lemma \ref{lem:Giusti1}.\\
Combining \eqref{i11l}, \eqref{jaydefl} and \eqref{i12l} we obtain the estimate of $|I_1|$, i.e.
\begin{align}
\label{i1l} 
|I_1|&\leq 3\varepsilon\int_\Omega\zeta^2|\tau_{j,h}H_{p/2}(Dv)|^2dx
+C_\varepsilon|h|^2\int_{B_R}(1+|Dv|^p)dx.
\end{align}
We now proceed with the estimate of $|I_3|$. Using \eqref{ipoal}, that $|D\zeta|\leq \frac{C}{R-r}$ and H\"older's inequality, we get
\begin{align}
\label{i3l}
\nonumber
|I_3|&\leq 2\int_\Omega\zeta |D\zeta||\tau_{j,h}(a(x))|(|Dv|-1)^{p-1}_+|\tau_{j,h}v|\textbf{ }dx\\\nonumber
&\leq \frac{C|h|}{R-r}\int_{B_{R/2}}(|Dv|-1)^{p-1}_+|\tau_{j,h}v|dx\\ \nonumber
&\leq \frac{C|h|^2}{R-r}\bigg(\int_{B_R}|Dv|^pdx\bigg)^{\frac{1}{p}}\bigg(\int_{B_R}(|Dv|-1)^{p}_+dx\bigg)^{\frac{p-1}{p}}\\
&\leq \frac{C|h|^2}{R-r}\int_{B_R}|Dv|^pdx,
\end{align}
where we also used the first estimate in Lemma \ref{lem:Giusti1}.\\
For the estimate of $|I_4|,$ we use that $|D\zeta|\leq\frac{C}{R-r}$, the second estimate in Lemma \ref{lem:Brasco} again, assumption $(ii)$ in \eqref{ipoa} and Young's inequality, as follows
\begin{align}
\label{i4l} \nonumber
|I_4|&\leq 2\int_\Omega\zeta |D\zeta| |a(x+he_j)||\tau_{j,h}(H_{p-1}(Dv))||\tau_{j,h}v|\textbf{ }dx\\ \nonumber
 &\leq \frac{C(M)}{R-r}\int_\Omega\zeta|\tau_{j,h}(H_{p/2}(Dv))|\Big||H_{p/2}(Dv(x+he_j))|^{\frac{p-2}{p}}+|H_{p/2}(Dv(x))|^{\frac{p-2}{p}}\Big||\tau_{j,h}v|\textbf{ }dx \\  \nonumber
 &\leq \varepsilon\int_\Omega \zeta^2|\tau_{j,h}(H_{p/2}(Dv))|^2dx\\ \nonumber
 &+\frac{C_\varepsilon(M)}{R-r}\int_{B_{R/2}} \Big||H_{p/2}(Dv(x+he_j))|^{\frac{p-2}{p}}+|H_{p/2}(Dv(x))|^{\frac{p-2}{p}}\Big|^2|\tau_{j,h}v|^2\textbf{ }dx\\\nonumber
 &\leq \varepsilon\int_\Omega \zeta^2|\tau_{j,h}(H_{p/2}(Dv))|^2dx\\ \nonumber
 &+ \frac{C_\varepsilon(M)}{R-r}\bigg(\int_{B_{R/2}}|\tau_{j,h}v|^pdx\bigg)^\frac{2}{p}\bigg(\int_{B_{R/2}}\Big||H_{p/2}(Dv(x+he_j))|^{\frac{p-2}{p}}+|H_{p/2}(Dv(x))|^{\frac{p-2}{p}}\Big|^{\frac{2p}{p-2}}dx\bigg)^{\frac{p-2}{p}} \\ \nonumber
 &\leq \varepsilon\int_\Omega \zeta^2|\tau_{j,h}(H_{p/2}(Dv))|^2dx\\ \nonumber
&+ \frac{C_\varepsilon(M)}{R-r}|h|^2\bigg(\int_{B_R}|Dv|^pdx\bigg)^{\frac{2}{p}}\bigg(\int_{B_R}|H_{p/2}(Dv(x))|^{2}dx\bigg)^{\frac{p-2}{p}}\\  
&\leq \varepsilon\int_\Omega \zeta^2|\tau_{j,h}(H_{p/2}(Dv))|^2dx+ \frac{C_\varepsilon(M)}{R-r}|h|^2\int_{B_{R}}(1+|Dv|^p)dx,
\end{align}
where we also used H\"older's inequality, Lemma \ref{lem:Giusti1} and the definition of $H_{p/2}(Dv)$.\\
We now proceed to estimate $|I_5|$, arguing as follows
\begin{align*}
    |I_5|&\leq \int_\Omega \zeta^2 |\tau_{j,h}v|\Big|(b(x+h,w(x+h))-b(x,w(x+h)))\Big|\textbf{ }dx\\ 
    &\leq L_2|h|\int_{B_{R/2}}|\tau_{j,h}v||w(x+h)|^qdx \\ 
    &\leq L_2|h|^2\bigg(\int_{B_{R}}|Dv|^pdx\bigg)^{\frac{1}{p}}\bigg(\int_{B_R}|w|^{\frac{qp}{p-1}}dx\bigg)^{\frac{p-1}{p}},
\end{align*}
where we used assumption \eqref{ipobforte1}, H\"older's inequality and Lemma \ref{lem:Giusti1}. Note that the use of H\"older's inequality is legitimate since $$\frac{qp}{p-1}<p^*\iff q<\frac{p^*(p-1)}{p}$$ which is true by assumption \eqref{q}. Therefore, we can estimate $|I_5|$ further by the use of Sobolev embedding theorem thus getting
\begin{align}
 \label{i5l} \nonumber   
   |I_5| &\leq C(n, p, q,L_2,R)|h|^2\bigg(\int_{B_{R}}|Dv|^pdx\bigg)^{\frac{1}{p}}\bigg(\int_{B_R}|w|^{p^*}dx\bigg)^{\frac{q}{p^*}}\\\nonumber
    &\leq C(n, p, q,L_2,R)|h|^2 \bigg(\int_{B_{R}}|Dv|^pdx\bigg)^{\frac{1}{p}}\bigg(\int_{B_R}(|Dw|^p+|w|^p)dx\bigg)^{\frac{q}{p}}\\
    &\leq C(n, p, q,L_2,R)|h|^2 \bigg(\int_{B_{R}}(1+|Dv|^p)dx\bigg)+C|h|^2\bigg(\int_{B_R}(|Dw|^p+|w|^p)dx\bigg)^{\frac{q}{p-1}},
\end{align} 
where we used Young's inequality in the last line.
To estimate $|I_6|$ we use assumptions \eqref{ipobforte2}, \eqref{q} and H\"older's inequality for $p>2$, that is legitimate since 
$$\frac{(q-1)p}{p-2}\leq p^* 
\iff q\leq\frac{p^*(p-2)}{p}+1.$$ We obtain
\begin{align}
\label{i6l}\nonumber
|I_6|&\leq \int_\Omega\zeta^2|\tau_{j,h}v||b(x,w(x+h))-b(x,w(x))|dx\\\nonumber
&\leq L_3\int_{B_{R/2}}|\tau_{j,h}v||\tau_{j,h}w|\bigg(|w(x)|^{q-1}+|w(x+h)|^{q-1}\bigg)dx\\\nonumber
&\leq C(n, p, q,L_3, R)|h|^2\bigg(\int_{B_R}|Dv|^pdx\bigg)^{\frac{1}{p}}\bigg(\int_{B_R}|Dw|^pdx\bigg)^{\frac{1}{p}}\bigg(\int_{B_R}|w|^{\frac{(q-1)p}{p-2}}dx\bigg)^{\frac{p-2}{p}}\\ 
\nonumber    
&\leq  C(n, p, q,L_3, R)|h|^2\bigg(\int_{B_R}|Dv|^pdx\bigg)^{\frac{1}{p}}\bigg(\int_{B_R}(|Dw|^p+|w|^p)dx\bigg)^{\frac{1}{p}}\bigg(\int_{B_R}|w|^{p^*}dx\bigg)^{\frac{q-1}{p^*}}\\\nonumber
&\leq C(n, p, q,L_3, R)|h|^2\bigg(\int_{B_R}|Dv|^pdx\bigg)^{\frac{1}{p}}\bigg(\int_{B_R}(|Dw|^p+|w|^p)dx\bigg)^{\frac{1}{p}}\bigg(\int_{B_R}(|Dw|^p+|w|^{p})dx\bigg)^{\frac{q-1}{p}}\\\nonumber
&\leq C(n, p, q,L_3, R)|h|^2\bigg(\int_{B_R}|Dv|^pdx\bigg)^{\frac{1}{p}}\bigg(\int_{B_R}(|Dw|^p+|w|^{p})dx\bigg)^{\frac{q}{p}}\\
&\leq C(n, p, q,L_3, R)|h|^2\bigg(\int_{B_R}(1+|Dv|^p)dx\bigg)+C|h|^2\bigg(\int_{B_R}(|Dw|^p+|w|^{p})dx\bigg)^{\frac{q}{p-1}},
\end{align}
where we used again the first estimate in Lemma \ref{lem:Giusti1}, Sobolev embedding theorem and Young's inequality. Let us point out that the previous estimate has been derived under the assumption $p>2$, which allows the use of Hölder's inequality with exponent $\frac{p}{p-2}$. The case $p=2$ requires a separate, yet simpler, analysis. Indeed, when $p=2$, assumption \eqref{q} forces $q=1$. Consequently, the term $(|w(x)|^{q-1} + |w(x+h)|^{q-1})$ reduces to the constant $2$, bypassing the need for the critical exponent. In this case, applying the standard Cauchy-Schwarz inequality and Lemma \ref{lem:Giusti1} directly yields:$$|I_6| \le 2 L_3 \int_{B_{R/2}} |\tau_{j,h} v| |\tau_{j,h} w| \, dx \le C(n,L_3,R) |h|^2 \left( \int_{B_R} |Dv|^2 \, dx \right)^{\frac{1}{2}} \left( \int_{B_R} |Dw|^2 \, dx \right)^{\frac{1}{2}}$$By applying Young's inequality, we immediately recover the desired bound:$$|I_6| \le C |h|^2 \int_{B_R} (1 + |Dv|^2) \, dx + C |h|^2 \int_{B_R} ( |w|^2 + |Dw|^2) \, dx,$$ which is precisely estimate \eqref{i6l} with $p=2$ and $q=1.$
At this point we insert estimates \eqref{i2l}, \eqref{i1l}, \eqref{i3l}, \eqref{i4l}, \eqref{i5l} and \eqref{i6l} in \eqref{dis:genl}, to get
\begin{align*}
&c_0\int_\Omega\zeta^2|\tau_{j,h}(H_{p/2}(Dv))|^2dx\leq 4\varepsilon\int_\Omega\zeta^2|\tau_{j,h}H_{p/2}(Dv)|^2dx\\
&+|h|^2\frac{C_\varepsilon}{{R-r}}\int_{B_R}(1+|Dv|^p)dx+C|h|^2\bigg(\int_{B_R}|Dw|^p+|w|^pdx\bigg)^{\frac{q}{p-1}}.
\end{align*}
Choosing $\varepsilon=\frac{c_0}{8}$, we can reabsorb the first integral in the right hand side of the previous estimate by the left hand side. Using also that $\zeta=1$ in $B_{r/2}$, dividing both sides of inequality by $|h|^2$ and letting $h$ go to $0$, we get
\begin{align}
\label{finel}
&\int_{B_{r/2}}|D(H_{p/2}(Dv))|^2dx\leq C\int_{B_R}(1+|Dv|^p)dx+C\bigg(\int_{B_R}(|w|^p+|Dw|^p)dx\bigg)^{\frac{q}{p-1}},
\end{align}
where $r<R$ and $C=C(n, m, M, L_1, L_2, L_3, p,q, r, R)$, i.e. the conclusion.
\end{proof}
In particular, as a Corollary of Theorem \ref{teo1}, we have a first regularity result for the solution $u$ of \eqref{equazione mia} under the same regularity assumptions on $a(x)$ and $b(x,u)$. Indeed, the same proof works in particular if $v=w=u$ and this yields 
\begin{cor}
Assume that $a(x)$ and $b(x,u)$ satisfy \eqref{ipoal}, \eqref{ipobforte1} and \eqref{ipobforte2} and that $q$ satisfies \eqref{q}. If $u\in W^{1,p}_{loc}(\Omega)$ solves \eqref{equazione mia}, then $H_{p/2}(Du)\in W^{1,2}_{loc}(\Omega)$ holds true with the estimate
\begin{align*}
&\int_{B_\rho}|D(H_{p/2}(Du))|^2dx\leq C\int_{B_R}(1+|Du|^p)dx+C\bigg(\int_{B_R}|u|^p+|Du|^pdx\bigg)^{\frac{q}{p-1}}.
\end{align*}
for all $B_\rho \subseteq B_{R/2}\subset B_R\Subset\Omega$ and $C=C(n, m, M, L_1, L_2, L_3, p,q,  \rho, R).$    
\end{cor} 
For further needs, we consider the following Theorem
\begin{thm}
\label{corol}
   Let $\varepsilon>0$ and $w\in W^{1,p}(\Omega).$ Let also $v_\varepsilon \in W^{1,p}_{loc}(\Omega)$ be a local solution of the following equation, 
\begin{equation}
\label{eq:eps}
    \mathrm{div}\bigg(a(x)H_{p-1}(Dv_\varepsilon)+\varepsilon|Dv_\varepsilon(x)|^{p-2}Dv_\varepsilon\bigg) = b(x,w).
\end{equation} 
Assume that $a(x)$ and $b(x,w)$ satisfy \eqref{ipoal}, \eqref{ipobforte1} and \eqref{ipobforte2} and that $q$ satisfies \eqref{q}, then $H_{p/2}(Dv_\varepsilon)\in W^{1,2}_{loc}(\Omega)$ with the estimate
\begin{align*}
&\int_{B_\rho}|D(H_{p/2}(Dv_\varepsilon))|^2dx\leq C\int_{B_R}(1+|Dv_\varepsilon|^p)dx+C\bigg(\int_{B_R}|w|^p+|Dw|^pdx\bigg)^{\frac{q}{p-1}}.
\end{align*}
for all $B_\rho \subseteq B_{R/2}\subset B_R\Subset\Omega$ and $C=C(m, M, L_1, L_2, L_3, n, p, q,  \rho, R),$ 
\end{thm}
\begin{proof} Since $v_\varepsilon$ solves equation \eqref{eq:eps} we have $$-\int_\Omega\biggl\langle a(x)H_{p-1}(Dv_\varepsilon)+\varepsilon |Dv_\varepsilon|^{p-2}Dv_\varepsilon,D\varphi\biggr\rangle\textbf{ }dx=\int_\Omega b(x,w)\varphi\textbf{ } dx,$$ 
choosing as test function the same as in \eqref{int} with the same bound on $|h|$ we obtain, 
\begin{align*}
&0=\int_\Omega\zeta^2\bigl\langle\tau_{j,h}(a(x))H_{p-1}(Dv_\varepsilon),\tau_{j,h}(Dv_\varepsilon)\bigr\rangle\textbf{ }dx\\
&+\int_\Omega \zeta^2a(x+he_j)\bigl\langle\tau_{j,h}(H_{p-1}(Dv_\varepsilon)),\tau_{j,h}(Dv_\varepsilon)\bigr\rangle\textbf{ }dx\\
&+2\int_\Omega\zeta\bigl\langle\tau_{j,h}(a(x))H_{p-1}(Dv_\varepsilon),\tau_{j,h}v_\varepsilon D\zeta\bigr\rangle\textbf{ }dx\\
&+2\int_\Omega \zeta a(x+he_j)\bigl\langle\tau_{j,h}(H_{p-1}(Dv_\varepsilon)),\tau_{j,h}v_\varepsilon D\zeta\bigr\rangle\textbf{ }dx\\
&+\int_\Omega \zeta^2 \tau_{j,h}v_\varepsilon[b(x+he_j,w(x+he_j))-b(x,w(x+he_j))]\textbf{ }dx\\
&+\int_\Omega \zeta^2 \tau_{j,h}v_\varepsilon[b(x,w(x+he_j))-b(x,w(x))]\textbf{ }dx\\
&+\int_\Omega  \zeta^2\bigl\langle\tau_{j,h}(\varepsilon |Dv_\varepsilon|^{p-2}Dv_\varepsilon),\tau_{j,h}Dv_\varepsilon \bigr\rangle\textbf{ }dx\\
&+2\int_\Omega\zeta  \bigl\langle\tau_{j,h}(\varepsilon |Dv_\varepsilon|^{p-2}Dv_\varepsilon),D\zeta\tau_{j,h}v_\varepsilon \bigr\rangle\textbf{ }dx\\
& =:I_1+I_2+I_3+I_4+I_5+I_6+I_7+I_8,
\end{align*}
that yields 
\begin{equation}
\label{dis:genl2}
    I_2+I_7\leq |I_1|+|I_3|+|I_4|+|I_5|+|I_6|+|I_8|.
\end{equation}  
Note that the integrals $I_j$ with $j=1,2,...,6$ are the exactly those in \eqref{dis:genl} with $v_\varepsilon$ in place of $v$, so we only need to estimate $I_7$ and $I_8.$   
By second estimate in Lemma \ref{lem:Lind}, it follows
\begin{equation}
    \label{i7}
    I_7\geq c\varepsilon\int_\Omega \zeta^2 \bigg||Dv_\varepsilon(x+he_j)|^{\frac{p-2}{2}}Dv_\varepsilon(x+he_j)-|Dv_\varepsilon|^{\frac{p-2}{2}}Dv_\varepsilon\bigg|^2dx= c\varepsilon\int_\Omega\zeta^2|\tau_{j,h}(V_p(Dv_\varepsilon))|^2dx,
\end{equation}
where $V_p(\xi)$ has been defined in \eqref{vp}\\
To estimate $I_8$ we use the properties of $D\zeta$, Lemma \ref{D1} twice with $\xi=Dv_\varepsilon(x+h)$ and $ \eta=Dv_\varepsilon(x)$, to get
\begin{align}
    \label{i8} \nonumber
    |I_8|&\leq \frac{C\varepsilon}{R-r}\int_\Omega\zeta|\tau_{j,h}(|Dv_\varepsilon|^{p-2}Dv_\varepsilon)||\tau_{j,h} v_\varepsilon| dx\\ \nonumber
    &\leq \frac{C\varepsilon}{R-r}\int_{\Omega} \zeta|\tau_{j,h}(Dv_\varepsilon)|\bigg(|Dv_\varepsilon(x+he_j)|^2+|Dv_\varepsilon(x)|^2\bigg)^{\frac{p-2}{2}}|\tau_{j,h}v_\varepsilon| \ dx\\ \nonumber
    &\leq \frac{C\varepsilon}{R-r}\int_{\Omega} \zeta|\tau_{j,h}(Dv_\varepsilon)|\bigg(|Dv_\varepsilon(x+he_j)|^2+|Dv_\varepsilon(x)|^2\bigg)^{\frac{p-2}{4}}\!\!\bigg(|Dv_\varepsilon(x+he_j)|^2+|Dv_\varepsilon(x)|^2\bigg)^{\frac{p-2}{4}}\!|\tau_{j,h}v_\varepsilon| \!\ dx\\ \nonumber
    & \leq\frac{C\varepsilon}{R-r}\int_{\Omega} \zeta|\tau_{j,h}(V_p(Dv_\varepsilon))|\bigg(|Dv_\varepsilon(x+he_j)|^2+|Dv_\varepsilon(x)|^2\bigg)^{\frac{p-2}{4}}|\tau_{j,h}v_\varepsilon| \ dx\\ \nonumber
    & \leq \frac{C\varepsilon}{R-r}\bigg(\int_{\Omega} \zeta^2|\tau_{j,h}(V_p(Dv_\varepsilon))|^2dx\bigg)^{\frac{1}{2}}\bigg(\int_{B_{R/2}}\!\!\bigg(|Dv_\varepsilon(x+he_j)|^2+|Dv_\varepsilon(x)|^2\bigg)^{\frac{p}{2}}dx\bigg)^{\frac{p-2}{2p}}\!\!\bigg(\int_\Omega|\tau_{j,h}v_\varepsilon|^pdx\bigg)^{\frac{1}{p}}\\ \nonumber
    & \leq |h|\frac{C\varepsilon}{R-r}\bigg(\int_{\Omega} \zeta^2|\tau_{j,h}(V_p(Dv_\varepsilon))|^2dx\bigg)^{\frac{1}{2}}\bigg(\int_{B_R}|Dv_\varepsilon(x)|^pdx\bigg)^{\frac{1}{2}}\\
    & \leq \theta\varepsilon\int_{\Omega} \zeta^2|\tau_{j,h}(V_p(Dv_\varepsilon))|^2dx+|h|^2\frac{C_\theta\varepsilon}{(R-r)^2}\int_{B_R}|Dv_\varepsilon(x)|^pdx
\end{align}  where we also used in the case $p>2$, in the third last line H\"older's inequality with exponents $$\frac{1}{2}+\frac{1}{p}+\frac{p-2}{2p}=1,$$ and in the second last line Lemma \ref{lem:Giusti1} and Young's inequality where $\theta$ will be chosen later. Observe that if $p=2$ $$|I_8|=2\varepsilon\int_\Omega\zeta |D\zeta|  |\tau_{j,h}( Dv_\varepsilon)||\tau_{j,h}v_\varepsilon|\textbf{ }dx,$$ that may be easily estimated using H\"older's inequality with exponents $$\frac{1}{2}+\frac{1}{2}=1.$$ Combining \eqref{i2l}, \eqref{i1l}, \eqref{i3l}, \eqref{i4l}, \eqref{i5l}, \eqref{i6l}, \eqref{i7} and \eqref{i8}, choosing $\theta$ appropriately and arguing as in the last part of the proof of Theorem \ref{teo1}, we obtain 
\begin{align*}
&\varepsilon\int_{B_\rho}|D(V_p(Dv_\varepsilon))|^2dx+\int_{B_\rho}|D(H_{p/2}(Dv_\varepsilon))|^2dx\leq C\int_{B_R}(1+|Dv_\varepsilon|^p)dx+C\bigg(\int_{B_R}|w|^p+|Dw|^pdx\bigg)^{\frac{q}{p-1}}.
\end{align*} and so 
\begin{align*}
\int_{B_\rho}|D(H_{p/2}(Dv_\varepsilon))|^2dx\leq C\int_{B_R}(1+|Dv_\varepsilon|^p)dx+C\bigg(\int_{B_R}|w|^p+|Dw|^pdx\bigg)^{\frac{q}{p-1}}.
\end{align*}
\end{proof}
\section{A priori estimate}
This section is devoted to the proof of an a priori estimate for the solutions to equation \eqref{equazione nuova} under the weaker assumptions \eqref{q} and \eqref{ipob1}, which is the first of the main points in our proof. More precisely, we assume a priori that $H_{p/2}(Dv)\in W^{1,2}_{loc}(\Omega)$ and we estimate the $L^2$ norm of the derivative of $H_{p/2}(Dv)$ with a constant that depends only on the $W^{1,p}$ norm of $w$, the $L^s$ norm of $k$ and the $L^\gamma$ norm of $g.$ 
  
\begin{thm}
\label{teo2}
   Let $w\in W^{1,p}_{loc}(\Omega)$ be fixed and $v\in W^{1,p}_{loc}(\Omega)$ be a weak solution of the equation \eqref{equazione nuova}, under assumptions \eqref{ipoa}, \eqref{q}, \eqref{ipob1}, \eqref{s e gam 1} and \eqref{s e gam 2}. If $H_{p/2}(Dv)\in W^{1,2}_{loc}(\Omega)$ then there exists $R_0=R_0(Da,n,p)$ such that the following estimate 
\begin{align*}
&\int_{B_{R/2}}|D(H_{p/2}(Dv))|^2dx\leq C\int_{B_R}(1+|Dv|^p)dx\\
&+C\bigg(\int_{B_{R}}|w|^{p}+|Dw|^pdx\bigg)^{\frac{q}{p-1}}\bigg[\bigg(\int_{B_R}|k(x)|^sdx\bigg)^{\frac{p}{s(p-1)}}+\bigg(\int_{B_R}|g(x)|^\gamma dx\bigg)^{\frac{p}{\gamma(p-1)}}\bigg]  
\end{align*}
    holds for every $B_R\subset B_{R_0}\Subset\Omega$ and a constant $C=C(n, p, q, s,\gamma, m, M, R).$
\end{thm}
\noindent \begin{proof}[\bfseries{Proof}]
Our starting point is inequality \eqref{dis:genl} in Theorem \ref{teo1}. Here we fix a ball $B_R\Subset\Omega$, radii $R/2\leq \rho_1\leq r_1<r_2<\rho_2\leq R<1$ and let $\zeta \in C^\infty_0(B_{r_2})$ be a cut off function such that $0\leq\zeta\leq 1, \zeta=1$ in $B_{r_1}$ and $|D\zeta|<\frac{C}{r_2-r_1}.$ Inequality \eqref{dis:genl} obviously holds true with the same choice of the test function in \eqref{int}, with $h\in \mathbb{R}\setminus\{0\}$ and such that $|h|<\frac{\rho_2-r_2}{2}$.\\ 
For the estimate of $I_1$, we use \eqref{prodl}, thus getting

\begin{equation}
    |I_1|\leq|I_{1,1}|+|I_{1,2}|+|I_{1,3}|.
\end{equation}
where  $|I_{1,j}|,\textbf{ }j=1,2,3$ are defined in \eqref{i1i2i3}.
We now proceed estimating the integrals $|I_{1,j}|$.
\noindent
Arguing as we did in \eqref{i11l} we get 
\begin{align*}
|I_{1,1}|&\leq\varepsilon\int_\Omega\zeta^2\big|\tau_{j,h}H_{p/2}(Dv)\big|^2dx \\ 
&+C_\varepsilon\int_\Omega\Big(|\tau_{j,h}a(x)|^2\bigg|(|Dv(x+he_j)|-1)^p_++(|Dv(x)|-1)^p_+\bigg|\Big)dx  \\ 
&\leq\varepsilon\int_\Omega\zeta^2\big|\tau_{j,h}H_{p/2}(Dv)\big|^2dx \\ 
&+ C_\varepsilon\bigg(\int_{B_{r_2}}|\tau_{j,h}a|^ndx\bigg)^{\frac{2}{n}}\bigg(\int_{B_{r_2}}(|Dv(x+he_j)|-1)^{\frac{pn}{n-2}}_++(|Dv(x)|-1)^{\frac{pn}{n-2}}_+dx\bigg)^{\frac{n-2}{n}}\\ 
&\leq\varepsilon\int_\Omega\zeta^2\big|\tau_{j,h}H_{p/2}(Dv)\big|^2dx + C_\varepsilon|h|^2\bigg(\int_{B_{R}}|Da|^n\bigg)^\frac{2}{n}\bigg(\int_{B_{\rho_2}}(|Dv(x)|-1)^{\frac{pn}{n-2}}_+dx\bigg)^{\frac{n-2}{n}},
\end{align*}
where we used in turn, assumption $(i)$ in \eqref{ipoa}, H\"older's inequality, the properties of $\zeta$ and Lemma \ref{lem:Giusti1}. The a priori assumption $H_{p/2}(Dv)\in W^{1,2}_{loc}(\Omega)$ allows us to use Lemma \ref{lem:sob}, thus getting 
\begin{align}
\label{i11} \nonumber
|I_{1,1}|&\leq\varepsilon\int_\Omega\zeta^2\big|\tau_{j,h}H_{p/2}(Dv)\big|^2dx+ C_\varepsilon|h|^2\bigg(\int_{B_R}|Da|^ndx\bigg)^{\frac{2}{n}}\int_{B_{\rho_2}}|DH_{p/2}(Dv)|^2dx\\
&+C_\varepsilon|h|^2\bigg(\int_{B_R}|Da|^ndx\bigg)^{\frac{2}{n}}\int_{B_R}(1+|Dv|^p)dx.
\end{align}
Now we proceed with the estimate of $|I_{1,2}|$. Arguing as we did in \eqref{jaydefl}, we get
\begin{align}
\label{i12}\nonumber
|I_{1,2}|&\leq \varepsilon\int_\Omega\zeta^2|\tau_{j,h}H_{p/2}(Dv)|^2dx \\\nonumber
&+C_\varepsilon\bigg(\int_{B_{r_2}}|\tau_{j,h}(a(x))|^ndx\bigg)^{\frac{2}{n}}\bigg(\int_{B_{r_2}}\bigg(\bigg||H_{p/2}(Dv(x+he_j))|^{\frac{p-2}{p}}+|H_{p/2}(Dv(x))|^{\frac{p-2}{p}}\bigg|\bigg)^{\frac{2n}{n-2}}dx\bigg)^\frac{n-2}{n}\\ \nonumber
&\leq \varepsilon\int_\Omega\zeta^2|\tau_{j,h}H_{p/2}(Dv)|^2dx+ C_\varepsilon|h|^2\bigg(\int_{B_R}|Da|^ndx\bigg)^{\frac{2}{n}}\bigg(\int_{B_{\rho_2}}|H_{p/2}(Dv(x))|^{\frac{p-2}{p}\frac{2n}{n-2}}dx\bigg)^{\frac{n-2}{n}}\\ \nonumber
&\leq \varepsilon\int_\Omega\zeta^2|\tau_{j,h}H_{p/2}(Dv)|^2dx+ C_\varepsilon|h|^2\bigg(\int_{B_R}|Da|^ndx\bigg)^{\frac{2}{n}}\bigg(\int_{B_{\rho_2}}(|Dv|-1)_+^{\frac{(p-2)n}{n-2}}dx\bigg)^{\frac{n-2}{n}}\\ \nonumber
&\leq \varepsilon\int_\Omega\zeta^2|\tau_{j,h}H_{p/2}(Dv)|^2dx+ C_\varepsilon|h|^2\bigg(\int_{B_R}|Da|^ndx\bigg)^{\frac{2}{n}}\bigg(\int_{B_{\rho_2}}\bigg(1+(|Dv|-1)_+^{\frac{pn}{n-2}}\bigg)dx\bigg)^{\frac{n-2}{n}}\\ \nonumber
&\leq \varepsilon\int_\Omega\zeta^2|\tau_{j,h}H_{p/2}(Dv)|^2dx+C_\varepsilon|h|^2\bigg(\int_{B_R}|Da|^ndx\bigg)^{\frac{2}{n}}\\
&+ C_\varepsilon|h|^2\bigg(\int_{B_R}|Da|^ndx\bigg)^{\frac{2}{n}}\bigg(\int_{B_{\rho_2}}|DH_{p/2}(Dv)|^2dx+\int_{B_{\rho_2}}(1+|Dv|^p)dx\bigg),
\end{align}
where we used Lemma \ref{lem:Brasco}, assumption \eqref{ipoa}, H\"older's inequality and Lemma \ref{lem:sob}.
Now we estimate $|I_{1,3}|$, arguing as in \eqref{i12l} as follows
\begin{align}
\label{i13} \nonumber
|I_{1,3}|
&\leq\varepsilon\int_\Omega \zeta^2|\tau_{j,h}H_{p/2}(Dv)|^2dx+ C_\varepsilon\int_\Omega\zeta^2|\tau_{j,h}a(x)|^2|Dv(x+he_j)|^pdx\\ \nonumber
&+C_\varepsilon\int_\Omega\zeta^2|\tau_{j,h}a(x)|^2\Big(|H_{p/2}(Dv(x+he_j))|^2+|H_{p/2}(Dv(x))|^2\Big)dx\\ \nonumber
&\leq\varepsilon\int_\Omega \zeta^2|\tau_{j,h}H_{p/2}(Dv)|^2dx+C_\varepsilon\bigg(\int_\Omega\zeta^2|\tau_{j,h}a(x)|^ndx\bigg)^{\frac{2}{n}}\bigg(\int_{B_{r_2}}|Dv(x+he_j)|^{\frac{np}{n-2}}dx\bigg)^{\frac{n-2}{n}}\\ \nonumber
&+C_\varepsilon\int_\Omega\zeta^2|\tau_{j,h}a(x)|^2\Big(|H_{p/2}(Dv(x+he_j))|^2+|H_{p/2}(Dv(x))|^2\Big)dx\\ \nonumber
&\leq\varepsilon\int_\Omega \zeta^2|\tau_{j,h}H_{p/2}(Dv)|^2dx+C_\varepsilon|h|^2\bigg(\int_{B_R}|Da|^ndx\bigg)^{\frac{2}{n}}\int_{B_{\rho_2}}|DH_{p/2}(Dv)|^2dx\\&+C_\varepsilon|h|^2\bigg(\int_{B_R}|Da|^ndx\bigg)^{\frac{2}{n}}\int_{B_R}(1+|Dv|^p)dx,
\end{align}
where we used Young's inequality, H\"older's inequality, Lemmas \ref{lem:Giusti1}, \ref{lem:sob} and in the last line, we argue as in \eqref{i11}.\\
Combining \eqref{i11}, \eqref{i12} and \eqref{i13}, we obtain
\begin{align}
\label{i1} \nonumber
|I_1|&\leq 3\varepsilon\int_\Omega\zeta^2|\tau_{j,h}H_{p/2}(Dv)|^2dx\\ \nonumber
&+C_\varepsilon|h|^2\bigg(\int_{B_R}|Da|^ndx\bigg)^{\frac{2}{n}}\bigg(\int_{B_R}(1+|Dv|^p)dx\bigg)\\
&+C_\varepsilon|h|^2\bigg(\int_{B_R}|Da|^ndx\bigg)^{\frac{2}{n}}\bigg(\int_{B_{\rho_2}}|DH_{p/2}(Dv)|^2dx\bigg).
\end{align}
We now proceed with the estimate of $|I_3|$, where $I_3$ has been defined in \eqref{dis:genl}. By the properties of $|D\zeta|$ we have
\begin{align*}
|I_3|&\leq 2\int_{B_{r_2}}|D\zeta||\tau_{j,h}a(x)||\tau_{j,h}v|(|Dv|-1)^{p-1}_+dx\\
&\leq \frac{C}{r_2-r_1}\int_{B_{r_2}}|\tau_{j,h}a(x)||\tau_{j,h}v|(|Dv|-1)^{p-1}_+dx.
\end{align*} 
Since $$\frac{1}{n}+\frac{1}{p}+\frac{(n-2)(p-1)}{np}\leq 1,$$ we can use H\"older's inequality with exponents $n, p, \frac{np}{(n-2)(p-1)}$ to deduce that
\begin{align}
\label{i3}\nonumber
|I_3|&\leq \frac{C(R)}{r_2-r_1}\bigg(\int_{B_{r_2}}|\tau_{j,h}a|^ndx\bigg)^{\frac{1}{n}}\bigg(\int_{B_{r_2}}|\tau_{j,h}v|^pdx\bigg)^{\frac{1}{p}}\bigg(\int_{B_{r_2}}(|Dv|-1)_+^{\frac{np}{n-2}}dx\bigg)^{\frac{(n-2)(p-1)}{np}}\\   \nonumber
&\leq \frac{C|h|^2}{r_2-r_1}\bigg(\int_{B_R}|Da|^ndx\bigg)^{\frac{1}{n}}\bigg(\int_{B_R}|Dv|^pdx\bigg)^{\frac{1}{p}}\bigg(\int_{B_{\rho_2}}|D(H_{p/2}(Dv))|^2dx+\int_{B_{\rho_2}}(1+|Dv|^p)dx\bigg)^{\frac{p-1}{p}}\\ \nonumber
&\leq\sigma|h|^2\int_{B_{\rho_2}}|D(H_{p/2}(Dv))|^2dx+\sigma|h|^2\int_{B_R}(1+|Dv|^p)dx\\ 
&+\frac{C_\sigma|h|^2}{(r_2-r_1)^p}\bigg(\int_{B_R}|Da|^ndx\bigg)^{\frac{p}{n}}\bigg(\int_{B_R}|Dv|^pdx\bigg),
\end{align}
where $\sigma>0$ will be chosen later, and we also used Lemmas \ref{lem:Giusti1}, \ref{lem:sob} and Young's inequality.\\
\noindent
To estimate $|I_5|$, we use assumption \eqref{s e gam 1} that allows us to apply H\"older's inequality with exponents $\frac{np}{n-2},$ $\frac{p^*}{q}$ and $s$, that is legitimate  by virtue of assumption \eqref{q}, where $s$ has been defined in \eqref{s e gam 1}.
\begin{align*}
|I_5|&\leq |h|\int_{B_{t}}|\tau_{j,h}v||(k(x+he_j)+k(x))||w(x+he_j)|^{q}dx\\
&\leq |h|\bigg(\int_{B_{t}}|\tau_{j,h}v|^{\frac{pn}{n-2}}dx\bigg)^{\frac{n-2}{pn}}\bigg(\int_{B_{\rho_2}}|w|^{p^*}dx\bigg)^{\frac{q}{p^*}}\bigg(\int_{B_{\rho_2}}|k(x)|^sdx\bigg)^{\frac{1}{s}}\\ 
&\leq|h|^2\bigg(\int_{B_{\rho_2}}|Dv|^{\frac{pn}{n-2}}dx\bigg)^{\frac{n-2}{pn}}\bigg(\int_{B_{\rho_2}}|w|^{p^*}dx\bigg)^{\frac{q}{p^*}}\bigg(\int_{B_{\rho_2}}|k(x)|^{s}dx\bigg)^{\frac{1}{s}},
\end{align*}
where we used assumption $(i)$ in \eqref{ipob1} and Lemma \ref{lem:Giusti1}. 
Using also Lemma \ref{lem:sob}, Young's inequality and Sobolev embedding theorem, we obtain
\begin{align}
\label{i5}\nonumber
|I_5|&\leq C|h|^2\bigg(\int_{B_{\rho_2}}|DH_{p/2}(Dv)|^2dx+\int_{B_{R}}(1+|Dv|^p)dx\bigg)^{\frac{1}{p}}\bigg(\int_{B_{R}}|w|^{p^*}dx\bigg)^{\frac{q}{p^*}}\bigg(\int_{B_{\rho_2}}|k(x)|^s dx\bigg)^{\frac{1}{s}}\\ \nonumber
&\leq\sigma|h|^2\int_{B_{\rho_2}}|DH_{p/2}(Dv)|^2dx+\sigma|h|^2\int_{B_{R}}(1+|Dv|^p)dx\\
&+C_\sigma|h|^2\bigg(\int_{B_{R}}|w|^{p}dx+\int_{B_{R}}|Dw|^pdx\bigg)^{\frac{q}{p-1}}\bigg(\int_{B_{\rho_2}}k(x)^s dx\bigg)^{\frac{p}{s(p-1)}}.
\end{align}
For the estimate of $|I_6|$, we use assumption \eqref{s e gam 2}, and we apply H\"older's inequality with exponents $\frac{np}{n-2},$ $p,$ $\frac{p^*}{q-1}$ and  $\gamma$, that is legitimate by assumption \eqref{q} and where $\gamma$ as been defined in \eqref{s e gam 2}. More precisely, we get
\begin{align}
\label{i6}
\nonumber
|I_6|&\leq \int_\Omega\zeta^2|\tau_{j,h}v||\tau_{j,h}w|2(|w(x+he_j)|+|w(x)|)^{q-1}|g(x)|dx \\ \nonumber
&\leq \bigg(\int_{B_{r_2}}|\tau_{j,h}v|^{\frac{pn}{n-2}}dx\bigg)^{\frac{n-2}{np}}\bigg(\int_{B_{r_2}}|\tau_{j,h}w|^{p}dx\bigg)^{\frac{1}{p}}\bigg(\int_{B_{\rho_2}}|w(x)|^{p^*}dx\bigg)^{\frac{q-1}{p^*}}\bigg(\int_{B_{\rho_2}} |g(x)|^\gamma dx \bigg)^{\frac{1}{\gamma}}\\ \nonumber
&\leq |h|^2\bigg(\int_{B_{\rho_2}}|Dv|^{\frac{pn}{n-2}}dx\bigg)^{\frac{n-2}{pn}}\bigg(\int_{B_{\rho_2}}|Dw|^{p}dx\bigg)^{\frac{1}{p}}\bigg(\int_{B_{\rho_2}}|w(x)|^{p^*}dx\bigg)^{\frac{q-1}{p^*}}\bigg(\int_{B_{\rho_2}}|g(x)|^\gamma dx\bigg)^{\frac{1}{\gamma}}\\\nonumber
&\leq |h|^2\bigg(\int_{B_{\rho_2}}|Dv|^{\frac{pn}{n-2}}dx\bigg)^{\frac{n-2}{pn}}\bigg(\int_{B_{\rho_2}}(|Dw|^{p}+|w(x)|^p)dx\bigg)^{\frac{q}{p}}\bigg(\int_{B_{\rho_2}}|g(x)|^\gamma dx\bigg)^{\frac{1}{\gamma}}\\\nonumber
&\leq\sigma|h|^2\int_{B_{\rho_2}}|DH_{p/2}(Dv)|^2dx+\sigma|h|^2\int_{B_{R}}(1+|Dv|^p)dx\\ 
&+C_\sigma|h|^2\bigg(\int_{B_{R}}(|w|^{p}+|Dw|^p)dx\bigg)^{\frac{q}{p-1}}\bigg(\int_{B_{\rho_2}}|g(x)|^\gamma dx\bigg)^{\frac{p}{\gamma(p-1)}},
\end{align}
where we used the properties of $\zeta$, Lemmas \ref{lem:Giusti1}, \ref{lem:sob} and Young's inequality. 
Similarly to what we observed in Section 3, the estimate \eqref{i6} formally holds for $p>2$, where the use of Hölder's inequality with the specified exponents is well-defined. In the limit case $p=2$, assumption \eqref{q} forces $q=1$, and the explicit dependency on the solution in the integrand vanishes. Furthermore, the integrability exponent $\gamma$ defined in \eqref{s e gam 2} reduces exactly to $\gamma = n$. Therefore, we can estimate $I_6$ by applying Hölder's inequality with exponents $\frac{2n}{n-2}$, $2$, and $n$:  $$|I_6| \leq 2 \int_{\Omega} \zeta^2 |\tau_{j,h} v| |\tau_{j,h} w| |g(x)| \, dx$$$$\leq 2 \left( \int_{B_{r_2}} |\tau_{j,h} v|^{\frac{2n}{n-2}} \, dx \right)^{\frac{n-2}{2n}} \left( \int_{B_{r_2}} |\tau_{j,h} w|^2 \, dx \right)^{\frac{1}{2}} \left( \int_{B_{\rho_2}} |g(x)|^n \, dx \right)^{\frac{1}{n}}$$$$\leq C |h|^2 \left( \int_{B_{\rho_2}} |Dv|^{\frac{2n}{n-2}} \, dx \right)^{\frac{n-2}{2n}} \left( \int_{B_{\rho_2}} |Dw|^2 \, dx \right)^{\frac{1}{2}} \left( \int_{B_{\rho_2}} |g(x)|^n \, dx \right)^{\frac{1}{n}}$$By combining Lemma \ref{lem:Giusti1} and Lemma \ref{lem:sob}, and applying Young's inequality, we obtain
\begin{align*}
|I_6| &\le \sigma |h|^2 \int_{B_{\rho_2}} |D(H_1(Dv))|^2 \, dx + \sigma |h|^2 \int_{B_R} (1+|Dv|^2) \, dx \\&+ C_\sigma |h|^2 \int_{B_R} (|w|^2 + |Dw|^2) \, dx \left( \int_{B_{\rho_2}} |g(x)|^n \, dx \right)^{\frac{2}{n}},    
\end{align*}
which is precisely estimate \eqref{i6} with $p=2$ and $q=1.$
Combining \eqref{i5} and \eqref{i6} we obtain
\begin{align}
\label{i5+i6}\nonumber
|I_5|+|I_6|&\leq 2\sigma|h|^2\int_{B_{\rho_2}}|DH_{p/2}(Dv)|^2dx+2\sigma|h|^2\int_{B_{R}}(1+|Dv|^p)dx\\\nonumber
&+C_\sigma|h|^2\bigg(\int_{B_{R}}|w|^{p}dx+\int_{B_R}|Dw|^pdx\bigg)^{\frac{q}{p-1}}\bigg(\int_{B_R}|k(x)|^sdx\bigg)^{\frac{p}{s(p-1)}}\\ 
&+C_\sigma|h|^2\bigg(\int_{B_{R}}|w|^{p}dx+\int_{B_R}|Dw|^pdx\bigg)^{\frac{q}{p-1}}\bigg(\int_{B_R}|g(x)|^\gamma dx\bigg)^{\frac{p}{\gamma(p-1)}}.
\end{align}The estimates of $I_2$ and $I_4$ are the same as in the previous Theorem at \eqref{i2l} and \eqref{i4l} respectively.
Next, we insert the estimates \eqref{i1}, \eqref{i2l}, \eqref{i3}, \eqref{i4l} and \eqref{i5+i6} in \eqref{dis:genl}, to get
\begin{align}
\label{prefine}\nonumber
&c_0\int_\Omega\zeta^2|\tau_{j,h}(H_{p/2}(Dv))|^2dx\leq 4\varepsilon\int_\Omega\zeta^2|\tau_{j,h}H_{p/2}(Dv)|^2dx \\\nonumber 
&+|h|^2\bigg(C_\varepsilon\bigg(\int_{B_R}|Da|^ndx\bigg)^{\frac{2}{n}}+3\sigma\bigg)\int_{B_{\rho_2}}|D(H_{p/2}(Dv))|^2dx\\ \nonumber
&+\frac{C_\sigma|h|^2}{(r_2-r_1)^p}\bigg(\int_{B_R}|Da|^ndx\bigg)^{\frac{p}{n}}\bigg(\int_{B_R}|Dv|^pdx\bigg)\\ \nonumber
&+|h|^2\bigg(\frac{C_\varepsilon}{r_2-r_1}+3\sigma+C\bigg(\int_{B_R}|Da|^ndx\bigg)^{\frac{2}{n}}\bigg)\int_{B_R}(1+|Dv|^p)dx\\\nonumber
&+C_\sigma|h|^2\bigg(\int_{B_{R}}|w|^{p}dx+\int_{B_R}|Dw|^pdx\bigg)^{\frac{q}{p-1}}\bigg(\int_{B_R}|k(x)|^sdx\bigg)^{\frac{p}{s(p-1)}}\\ 
&+C_\sigma|h|^2\bigg(\int_{B_{R}}|w|^{p}dx+\int_{B_R}|Dw|^pdx\bigg)^{\frac{q}{p-1}}\bigg(\int_{B_R}|g(x)|^\gamma dx\bigg)^{\frac{p}{\gamma(p-1)}}.
\end{align}
Choosing $\varepsilon=\frac{c_0}{8}$, we can reabsorb the first integral in the right hand side of the previous estimate by the left hand side. Using also that $\zeta=1$ in $B_{r_1}$, dividing both sides of inequality by $|h|^2$ and letting $h$ go to $0$, by the a priori assumption $H_{p/2}(Du)\in W^{1,2}_{loc}(\Omega),$ we get
\begin{align*}
&\frac{c_0}{2}\int_{B_{\rho_1}}|D(H_{p/2}(Dv))|^2dx\leq\frac{c_0}{2}\int_{B_{r_1}}|D(H_{p/2}(Dv))|^2dx\\
&\leq \bigg(\tilde{C}\bigg(\int_{B_R}|Da|^ndx\bigg)^{\frac{2}{n}}+3\sigma\bigg)\int_{B_{\rho_2}}|D(H_{p/2}(Dv))|^2dx\\
&+\bigg(\frac{C}{r_2-r_1}+3\sigma+C\bigg(\int_{B_R}|Da|^ndx\bigg)^{\frac{2}{n}}\bigg)\int_{B_R}(1+|Dv|^p)dx\\
&+\frac{C}{(r_2-r_1)^p}\bigg(\int_{B_R}|Da|^ndx\bigg)^{\frac{p}{n}}\bigg(\int_{B_R}|Dv|^pdx\bigg)\\
&+C_\sigma\bigg(\int_{B_{R}}|w|^{p}dx+\int_{B_R}|Dw|^pdx\bigg)^{\frac{q}{p-1}}\bigg(\int_{B_R}|k(x)|^sdx\bigg)^{\frac{p}{s(p-1)}}\\ 
&+C_\sigma\bigg(\int_{B_{R}}|w|^{p}+|Dw|^pdx\bigg)^{\frac{q}{p-1}}\bigg(\int_{B_R}|g(x)|^\gamma dx\bigg)^{\frac{p}{\gamma(p-1)}}.
\end{align*} Dividing both sides by $\frac{c_0}{2}$, we obtain
\begin{align*}
&\int_{B_{\rho_1}}|D(H_{p/2}(Dv))|^2dx\\
&\leq \frac{2}{c_0}\bigg[\tilde{C}\bigg(\int_{B_R}|Da|^ndx\bigg)^{\frac{2}{n}}+3\sigma\bigg]\int_{B_{\rho_2}}|D(H_{p/2}(Dv))|^2dx\\
&+\bigg(\frac{C}{r_2-r_1}+\frac{6\sigma}{c_0}+C\bigg(\int_{B_R}|Da|^ndx\bigg)^{\frac{2}{n}}\bigg)\int_{B_R}(1+|Dv|^p)dx\\
&+\frac{C}{(r_2-r_1)^p}\bigg(\int_{B_R}|Da|^ndx\bigg)^{\frac{p}{n}}\bigg(\int_{B_R}|Dv|^pdx\bigg)\\
&+C_\sigma\bigg(\int_{B_{R}}|w|^{p}dx+\int_{B_R}|Dw|^pdx\bigg)^{\frac{q}{p-1}}\bigg(\int_{B_R}|k(x)|^sdx\bigg)^{\frac{p}{s(p-1)}}\\ 
&+C_\sigma\bigg(\int_{B_{R}}|w|^{p}+|Dw|^pdx\bigg)^{\frac{q}{p-1}}\bigg(\int_{B_R}|g(x)|^\gamma dx\bigg)^{\frac{p}{\gamma(p-1)}}.
\end{align*}
By the absolute continuity of the integral there exists $R_0$ such that $$\bigg(\int_{B_{R_0}}|Da|^ndx\bigg)^{\frac{2}{n}}<\frac{c_0}{8\tilde{C}},$$ hence, choosing $R<R_0$, we have
\begin{equation}
\frac{2}{c_0}\bigg[\tilde{C}\bigg(\int_{B_R}|Da|^ndx\bigg)^{\frac{2}{n}}+3\sigma\bigg]\leq \frac{1}{4}+\frac{2}{c_0}3\sigma,
\end{equation}
and so, choosing $\sigma=\frac{c_0}{24}$, we have that \begin{equation}
\frac{2}{c_0}\bigg[\tilde{C}\bigg(\int_{B_R}|Da|^ndx\bigg)^{\frac{2}{n}}+3\sigma\bigg]\leq \frac{1}{2}
\end{equation}
Moreover, choosing $r_2$ and $r_1$ such that $$\frac{1}{r_2-r_1}\cong\frac{1}{\rho_2-\rho_1},$$ estimate \eqref{prefine} can be written as follows 
\begin{align*}
\int_{B_{\rho_1}}|D(H_{p/2}(Dv))|^2dx&\leq \frac{1}{2}\int_{B_{\rho_2}}|D(H_{p/2}(Dv))|^2dx\\
&+\bigg(\frac{C}{\rho_2-\rho_1}+\frac{1}{2}\bigg)\int_{B_R}(1+|Dv|^p)dx+\frac{C}{(\rho_2-\rho_1)^p}\bigg(\int_{B_R}|Dv|^p\bigg)\\
&+C_\sigma\bigg(\int_{B_{R}}|w|^{p}dx+\int_{B_R}|Dw|^pdx\bigg)^{\frac{q}{p-1}}\bigg(\int_{B_R}|k(x)|^sdx\bigg)^{\frac{p}{s(p-1)}}\\ 
&+C_\sigma\bigg(\int_{B_{R}}|w|^{p}+|Dw|^pdx\bigg)^{\frac{q}{p-1}}\bigg(\int_{B_R}|g(x)|^\gamma dx\bigg)^{\frac{p}{\gamma(p-1)}}.  
\end{align*}
At this point we use the iteration Lemma \ref{lem:Giusti2} with $Z(\rho)=\int_{B_\rho}|D(H_{p/2}(Du))|^2dx$, that yields
\begin{align}
\label{fine}\nonumber
\int_{B_{R/2}}|D(H_{p/2}(Dv))|^2dx&\leq C\int_{B_R}(1+|Dv|^p)dx+C\bigg(\int_{B_{R}}|w|^{p}+|Dw|^pdx\bigg)^{\frac{q}{p-1}}\bigg(\int_{B_R}|k(x)|^sdx\bigg)^{\frac{p}{s(p-1)}}\\
&+C\bigg(\int_{B_{R}}|w|^{p}+|Dw|^pdx\bigg)^{\frac{q}{p-1}}\bigg(\int_{B_R}|g(x)|^\gamma dx\bigg)^{\frac{p}{\gamma(p-1)}}
\end{align} i.e. the conclusion.
\end{proof}In particular, as a corollary of Theorem \ref{teo2}, we have an a priori estimate for the solution $u$ of \eqref{equazione mia}. Indeed, the same proof works in particular if $v=w=u$ and this yields 
\begin{cor}
Let $u\in W^{1,p}_{loc}(\Omega)$ be a solution of \eqref{equazione mia}, under assumptions \eqref{ipoa}, \eqref{q}, \eqref{ipob1}, \eqref{s e gam 1} and \eqref{s e gam 2}. If $H_{p/2}(Du)\in W^{1,2}_{loc}(\Omega)$ then there exist $R_0=R_0(Da,n,p)$ such that the following estimate 
\begin{align*}
&\int_{B_{R/2}}|D(H_{p/2}(Du))|^2dx\leq C\int_{B_R}(1+|Du|^p)dx\\
&+C\bigg(\int_{B_{R}}|u|^{p}+|Du|^pdx\bigg)^{\frac{q}{p-1}}\bigg(\int_{B_R}|k(x)|^sdx\bigg)^{\frac{p}{s(p-1)}}\\
&+C\bigg(\int_{B_{R}}|u|^{p}+|Du|^pdx\bigg)^{\frac{q}{p-1}}\bigg(\int_{B_R}|g(x)|^\gamma dx\bigg)^{\frac{p}{\gamma(p-1)}}\bigg.,   
\end{align*}
    holds for every $B_R\subset B_{R_0}\Subset\Omega$ and a constant $C=C(n, p, q, s,\gamma, m, M, R).$
\end{cor}
\section{Proof of Theorem \ref{teo3}}
The aim of this section is to conclude the proof of Theorem \ref{teo3} by using a suitable approximation argument. More precisely, let $u\in W^{1,p}_{loc}(\Omega)$ be a local solution of \eqref{equazione mia}, fix a ball $B_R\Subset\Omega$ and assume without loss of generality $R\leq 1.$
For $\varepsilon\in(0,\delta),$ with $\delta$ such that $B_R\subset B_{R+\delta}\subset B_{R_0}\Subset\Omega,$ let $u_\varepsilon\in W^{1,p}(B_R)$ be the unique solution of the following problem
\begin{equation}
    \label{equ regolare}
    \left\{
    \begin{array}{lll}
    \mathrm{div}\bigg(a_\varepsilon(x)H_{p-1}(Du_\varepsilon)+\varepsilon|Du_\varepsilon(x)|^{p-2}Du_\varepsilon\bigg) = b_\varepsilon(x,u) & \text{in } B_R \\[6pt]
    u_\varepsilon  = u & \text{on } \partial B_R
        
    \end{array}
    \right.
\end{equation}
where
\begin{itemize}
\item $a_\varepsilon(x):=a(x)\ast \rho_\varepsilon$, 
    \item $b_\varepsilon:=b\ast\rho_\varepsilon$, i.e. $b_\varepsilon (x,s)=\int_{B_1}b(x+\varepsilon w,s)\cdot\rho(w)dw$
\end{itemize}
with $\{\rho_\varepsilon\}_{\varepsilon>0}$ a family of standard compactly supported $C^\infty$ mollifiers.\\
Observe that the function $b_\varepsilon(x,u)$ satisfies assumption $(i)$ in \eqref{ipob1}, with $k_\varepsilon$ in place of $k,$ indeed 
\begin{align}\nonumber
\label{5star}
|b_\varepsilon(x,s)-b_\varepsilon(y,s)|&=\bigg|\int_{B_1}\bigg(b(x+\varepsilon w,s)-b(y+\varepsilon w,s)\bigg)\cdot\rho(w)dw\bigg|\\ \nonumber
&\leq \int_{B_1}|b(x+\varepsilon w,s)-b(y+\varepsilon w,s)|\cdot\rho(w)dw\\ \nonumber
&\leq |x-y||s|^q\bigg[\int_{B_1}k(x+\varepsilon w)\cdot\rho(w)dw+\int_{B_1}k(y+\varepsilon w)\cdot\rho(w)dw\bigg]\\
&=|x-y|(k_\varepsilon(x)+k_\varepsilon(y))|s|^q,
\end{align} where $k_\varepsilon$ is the mollified of $k$. Hence, the function $b_\varepsilon(x,u)$ satisfies also \eqref{ipobforte1}, in fact
\begin{equation}
\label{prima condi}
|b_\varepsilon(x,s)-b_\varepsilon(y,s)|\leq L_{2,\varepsilon}|x-y||s|^q  
\end{equation}
where $L_{2,\varepsilon}=2||k_\varepsilon||_{L^\infty(B_R)}.$ 
The same observations can be applied to get assumption $(ii)$ in \eqref{ipob1}, that is
\begin{align}
\label{5starstar}
 \nonumber
|b_\varepsilon(x,s)-b_\varepsilon(x,t)|&=\bigg|\int_{B_1}\bigg(b(x+\varepsilon w,s)-b(x+\varepsilon w,t)\bigg)\cdot\rho(w)dw\bigg|\\ \nonumber
&\leq \int_{B_1}|b(x+\varepsilon w,s)-b(x+\varepsilon w,t)|\cdot\rho(w)dw\\\nonumber
&\leq |s-t|(|s|^{q-1}+|t|^{q-1})\int_{B_1}g(x+\varepsilon w)\cdot\rho(w)dw\\ 
&=|s-t|(|s|^{q-1}+|t|^{q-1})g_\varepsilon(x),
\end{align} where $g_\varepsilon$ is the mollified of $g$. Hence, the function $b_\varepsilon(x,u)$ satisfies also \eqref{ipobforte2}, in fact
\begin{equation}
\label{seconda condi}
|b_\varepsilon(x,s)-b_\varepsilon(x,t)| \leq L_{3,\varepsilon}|s-t|(|s|^{q-1}+|t|^{q-1}),  
\end{equation}
where $L_{3,\varepsilon}=||g_\varepsilon||_{L^\infty(B_R)}.$\\
Hence, by virtue of \eqref{5star} and \eqref{5starstar}, we can use Theorem \ref{corol} with $u_\varepsilon$ in place of $v_\varepsilon$ and $u$ in place of $w$ to deduce that 
\begin{align}
\label{approx}\nonumber
\int_{B_r}|D(H_{p/2}(Du_\varepsilon))|^2dx&\leq C\int_{B_\rho}(1+|Du_\varepsilon|^p)dx+C\bigg(\int_{B_\rho}|u|^{p}+|Du|^pdx\bigg)^{\frac{q}{p-1}}\bigg(\int_{B_\rho}|k_\varepsilon(x)|^sdx\bigg)^{\frac{p}{s(p-1)}}\\
&+C\bigg(\int_{B_\rho}|u|^{p}+|Du|^pdx\bigg)^{\frac{q}{p-1}}\bigg(\int_{B_\rho}|g_\varepsilon(x)|^\gamma dx\bigg)^{\frac{p}{\gamma(p-1)}},
\end{align}
for every $B_\rho\Subset B_R, \ r<\rho$, with a constant $C$ independent of $\varepsilon$. 
Note that, due to properties of mollifiers, we have

\begin{equation}
\label{ae}
 0<m<a_\varepsilon(x)<M \ \ \text{a.e. }x\in B_R,
\end{equation}where $m$ and $M$ are defined in \eqref{ipoa}, and also
\begin{equation}
\label{ke}
 ||k_\varepsilon(x)||_{L^s(B_R)}\leq C||k(x)||_{L^s(B_{R+\varepsilon})}\leq C||k(x)||_{L^s(B_{R+\delta})},
\end{equation}
\begin{equation}
\label{he}
||g_\varepsilon(x)||_{L^\gamma(B_R)}\leq C||g(x)||_{L^\gamma(B_{R+\varepsilon})}\leq C||g(x)||_{L^\gamma(B_{R+\delta})},
\end{equation} with $C$ indepent of $\varepsilon$ and where we used that $\varepsilon<\delta.$
Consequently, the norm of the functions $g_\varepsilon$ and $k_\varepsilon$ can be bounded independently of $\varepsilon.$  
Since $u_\varepsilon$ is a weak of solution of \eqref{equ regolare}, then 
\begin{equation}
   \label{intregolare}
   -\int_{B_R} \bigl\langle a_\varepsilon(x)H_{p-1}(Du_\varepsilon)+\varepsilon|Du_\varepsilon|^{p-2}Du_\varepsilon,D\varphi\bigr\rangle \textbf{ }dx=\int_{B_R} b_\varepsilon(x,u)\varphi\textbf{ } dx \ \textbf{ }\text{ for all $\varphi \in W^{1,p}_0(B_R)$}.
\end{equation}
Choosing $\varphi=u_\varepsilon-u$ in \eqref{intregolare}, we obtain \noindent
\begin{align*}
0&=\int_{B_R} \bigl\langle a_\varepsilon(x)H_{p-1}(Du_\varepsilon)+\varepsilon|Du_\varepsilon|^{p-2}Du_\varepsilon,Du_\varepsilon-Du\bigr\rangle \textbf{ }dx+\int_{B_R} b_\varepsilon(x,u)(u_\varepsilon-u)\textbf{ } dx\\
   &=\int_{B_R}\bigl\langle a_\varepsilon(x)H_{p-1}(Du_\varepsilon)+\varepsilon|Du_\varepsilon|^{p-2}Du_\varepsilon,Du_\varepsilon\bigr\rangle dx
-\int_{B_R} \bigl\langle a_\varepsilon(x)H_{p-1}(Du_\varepsilon)+\varepsilon|Du_\varepsilon|^{p-2}Du_\varepsilon,Du\bigr\rangle dx\\&+\int_{B_R} b_\varepsilon(x,u)(u_\varepsilon-u)\textbf{ } dx\\ 
&=:J_1-J_2+J_3,
\end{align*}
that yields 
\begin{equation}
\label{dis:app}
    J_1\leq |J_2|+|J_3|.
\end{equation}  
By \eqref{ae} and Lemma \ref{lem:Brasco}, we deduce that
\begin{align}
    \label{lhs}
    J_1&\geq c_0(m,n,p)\int_{B_R} |H_{p/2}(Du_\varepsilon)|^2dx+\varepsilon\int_{B_R}|Du_\varepsilon|^pdx.
\end{align}
In order to estimate $|J_2|$, we use again \eqref{ae}, Lemma \ref{lem:Brasco} and Young's inequality as follows
\begin{align}
    \label{rhs1}\nonumber
    J_2&\leq M\int_{B_R} |H_{p-1}(Du_\varepsilon)||Du| dx+\varepsilon\int_{B_R}|Du_\varepsilon|^{p-1}|Du|dx\\ 
    &\leq \sigma\int_{B_R} |H_{p/2}(Du_\varepsilon)|^2dx+C_{\sigma,\varepsilon}\int_{B_R} (1+|Du|^p)dx+\frac{\varepsilon}{p'}\int_{B_R}|Du_\varepsilon|^pdx,
\end{align}
where $\sigma>0$ will be chosen later.
Now we estimate $|J_3|$ using \eqref{ipb}, H\"older's inequality with exponents $p^*, \frac{p^*}{q}$ and $\gamma$, that is legitimate by virtue of assumption \eqref{q} and \eqref{s e gam 2}, Sobolev-Poincaré theorem and Sobolev embedding theorem, as follows 
\begin{align}
\label{rhs2}\nonumber
|J_3|&\leq \int_{B_R}|u_\varepsilon-u||b_\varepsilon(x,u)|dx\leq\int_{B_R} |u_\varepsilon-u||g_\varepsilon(x)|u|^qdx\\\nonumber
&\leq \bigg(\int_{B_R}|u_\varepsilon-u|^{p^*}dx\bigg)^{\frac{1}{p^*}}\bigg(\int_{B_R}g_\varepsilon(x)^\gamma dx\bigg)^{\frac{1}{\gamma}}\bigg(\int_{B_R}|u|^{p^*}dx\bigg)^{\frac{q}{p^*}}\\\nonumber
&\leq C(n,p,q)\bigg(\int_{B_R}|Du_\varepsilon-Du|^{p}dx\bigg)^{\frac{1}{p}}\bigg(\int_{B_R}g_\varepsilon(x)^\gamma dx\bigg)^{\frac{1}{\gamma}}\bigg(\int_{B_R}|u|^{p}dx+\int_{B_R}|Du|^{p}dx\bigg)^{\frac{q}{p}}\\ \nonumber
&\leq\sigma\bigg(\int_{B_R}|Du_\varepsilon|^p+|Du|^{p}dx\bigg)+C_\sigma\bigg(\int_{B_R}g_\varepsilon(x)^\gamma dx\bigg)^{\frac{p}{\gamma(p-1)}}\bigg(\int_{B_R}|u|^{p}dx+\int_{B_R}|Du|^{p}dx\bigg)^{\frac{q}{p-1}}\\ \nonumber
&\leq 2^{p-1}\sigma \int_{B_R}(|Du_\varepsilon|-1)_+^p dx +2^{p-1}\sigma \int_{B_R}(1+|Du|^p) dx \\ 
&+C_\sigma\bigg(\int_{B_R}g_\varepsilon(x)^\gamma dx\bigg)^{\frac{p}{\gamma(p-1)}}\bigg(\int_{B_R}|u|^{p}dx+\int_{B_R}|Du|^{p}dx\bigg)^{\frac{q}{p-1}},
\end{align} where we used also Young's inequality in the fourth line.
Inserting \eqref{lhs}, \eqref{rhs1} and \eqref{rhs2} in \eqref{dis:app} we obtain
\begin{align}
\label{m}
\nonumber
   &c_0\int_{B_R} |H_{p/2}(Du_\varepsilon)|^2dx+\varepsilon\int_{B_R}|Du_\varepsilon|^pdx\leq 2^p\sigma\int_{B_R} |H_{p/2}(Du_\varepsilon)|^2dx+C_\sigma\int_{B_R}(1+|Du|^p)dx\\
&+C_\sigma\bigg(\int_{B_R}g_\varepsilon(x)^\gamma dx\bigg)^{\frac{p}{\gamma(p-1)}}\bigg(\int_{B_R}|u|^{p}dx+\int_{B_R}|Du|^{p}dx\bigg)^{\frac{q}{p-1}}+\frac{\varepsilon}{p'}\int_{B_R}|Du_\varepsilon|^pdx,
\end{align}
where we used that $|H_{p/2}(Du_\varepsilon)|^2=(|Du_\varepsilon|-1)^p_+.$
Choosing $\sigma=\frac{c_0}{2^{p+1}}$ we can reabsorb the first integral in the right hand side by the left hand side, and reabsorbing also the last integral in \eqref{m}, we get
\begin{align}
\label{approx2} \nonumber
 c_0\int_{B_R} |H_{p/2}(Du_\varepsilon)|^2dx&+\frac{2\varepsilon}{p}\int_{B_R}|Du_\varepsilon|^pdx\leq C\int_{B_R}(1+|Du|^p)dx\\
&+C\bigg(\int_{B_R}g_\varepsilon(x)^\gamma dx\bigg)^{\frac{p}{\gamma(p-1)}}\bigg(\int_{B_R}|u|^{p}dx+\int_{B_R}|Du|^{p}dx\bigg)^{\frac{q}{p-1}},
\end{align}
and by virtue of \eqref{he}, inequality \eqref{approx2} implies that
\begin{align}
\label{c} \nonumber
\int_{B_R}|Du_\varepsilon|^pdx &\leq 2^{p-1} \int_{B_R}(|Du_\varepsilon|-1)^p_+dx+ 2^{p-1}|B_R|= 2^{p-1}\int_{B_R}|H_{p/2}(Du_\varepsilon)|^2dx+ 2^{p-1}|B_R|\\ \nonumber
&\leq C\int_{B_R}(1+|Du|^p)dx+C|B_R|\\
&+C\bigg(\int_{{B_{R+\delta}}}g(x)^\gamma dx\bigg)^{\frac{p}{\gamma(p-1)}}\bigg(\int_{B_R}|u|^{p}dx+\int_{B_R}|Du|^{p}dx\bigg)^{\frac{q}{p-1}},
\end{align}
with $C$ independent of $\varepsilon$.
Using \eqref{c} in the right hand side of \eqref{approx}, we get
\begin{align}
\label{approx3}\nonumber
\int_{B_\rho}|DH_{p/2}(Du_\varepsilon)|^2dx&\leq C\bigg(\int_{B_{R}}|u|^{p}+|Du|^pdx\bigg)^{\frac{q}{p-1}}\bigg(\int_{B_R}|k_\varepsilon(x)|^sdx\bigg)^{\frac{p}{s(p-1)}}\\ \nonumber
&+C\bigg(\int_{B_{R}}|u|^{p}+|Du|^pdx\bigg)^{\frac{q}{p-1}}\bigg(\int_{B_R}|g_\varepsilon(x)|^\gamma dx\bigg)^{\frac{p}{\gamma(p-1)}}\\ \nonumber
&+C\int_{B_R}(1+|Du|^p)dx \\ \nonumber
&\leq C||k(x)||_{L^s(B_{R+\delta})}^{\frac{p}{p-1}}\bigg(\int_{B_{R}}|u|^{p}+|Du|^pdx\bigg)^{\frac{q}{p-1}}\\ \nonumber
&+C||g(x)||_{L^\gamma(B_{R+\delta)}}^{\frac{p}{p-1}}\bigg(\int_{B_{R}}|u|^{p}+|Du|^pdx\bigg)^{\frac{q}{p-1}}\\
&+C\int_{B_R}(1+|Du|^p)dx,
\end{align}
where we used \eqref{ke} and \eqref{he}.\\ By \eqref{c} we have 
\begin{align}
    u_\varepsilon&\rightharpoonup \mathbf{v} \text{ in }W^{1,p}_{loc}(B_R),
    \end{align}
    moreover using \eqref{c} in the right hand side of \eqref{approx} we also get $$H_{p/2}(Du_\varepsilon) \rightharpoonup \mathbf{w}\text{ in } W^{1,2}_{loc}(B_R).$$   As consequence we have 
    \begin{equation}
    \label{conver}
      H_{p/2}(Du_\varepsilon) \to \mathbf{w} \text{ strongly} \text{ in } L^2_{loc}(B_R).  
    \end{equation} Our next aim is to prove that $\mathbf{w}$ is equal to $H_{p/2}(D\mathbf{v})$ so that $H_{p/2}(Du_\varepsilon) \to H_{p/2}(D\mathbf{v})$ strongly in $L^2_{loc}(B_R).$ To this end, we know that by \eqref{conver} we have the convergence also up to a subsequence, \[ H_{p/2}(Du_\varepsilon)\to \mathbf{w} \qquad\text{a.e. in }B_R. \] 
    We separately consider the sets
\[
E:=\{x\in B_R:w(x)\neq0\},
\qquad
Z:=\{x\in B_R:w(x)=0\}.
\]
On the set $E$, the map $H_{p/2}(\xi)$ is invertible, more precisely, if $\eta \not =0$ its explicit inverse takes the form $$H_{p/2}^{-1}(\eta)=(|\eta|^{\frac{2}{p}}+1)\frac{\eta}{|\eta|}$$
    it follows that \[ Du_\varepsilon\to H_{p/2}^{-1}(\mathbf{w}) \qquad\text{a.e. on }E. \] Furthermore, the sequence $\{Du_\varepsilon\}$ is equi-integrable by virtue of \eqref{c}, hence we can use Vitali's Theorem to deduce that 
    $$Du_\varepsilon \to H_{p/2}^{-1}(\mathbf{w}) \text{ strongly in } L^1_{loc}(E),$$
    on the other hand since $Du_\varepsilon$ weakly converges to $D\mathbf{v}$ in $L^p_{loc}(E),$ we obtain $$D\mathbf{v}=H_{p/2}^{-1}(\mathbf{w}).$$
    Consequently, by the uniqueness of the weak limit, it holds\[ \mathbf{w}=H_{p/2}(D\mathbf{v}) \qquad\text{a.e. in }E. \]
It remains to consider the set $Z$, let
\[
P(\xi):=
\begin{cases}
\xi, & |\xi|\leq1,\\[2mm]
\dfrac{\xi}{|\xi|}, & |\xi|>1,
\end{cases}
\]
be the projection onto the closed unit ball. Then
\[
|\xi-P(\xi)|=(|\xi|-1)_+,
\]
and
\[
|H_{p/2}(\xi)|^2=(|\xi|-1)_+^p
=|\xi-P(\xi)|^p.
\]
Hence, on $Z=\{\mathbf{w}=0\}$,
\[
\begin{aligned}
\int_Z |Du_\varepsilon-P(Du_\varepsilon)|^p\,dx
&=
\int_Z |H_{p/2}(Du_\varepsilon)|^2\,dx\\
&=
\int_Z
|H_{p/2}(Du_\varepsilon)-\mathbf{w}|^2\,dx
\longrightarrow0,
\end{aligned}
\]
and so by \eqref{conver}
\[
Du_\varepsilon-P(Du_\varepsilon)\to0
\qquad\text{strongly in }L^p_{loc}(Z).
\]
Since
\[
|P(Du_\varepsilon)|\leq1
\text{ a.e.}
\text{ and }
P(Du_\varepsilon)
=
Du_\varepsilon-
\bigl(Du_\varepsilon-P(Du_\varepsilon)\bigr),
\]
we deduce that
\[
P(Du_\varepsilon)\rightharpoonup D\mathbf{v}
\qquad\text{weakly in }L^p_{loc}(Z).
\]
By the weak lower semicontinuity of the $L^p$-norm, for every
measurable set $A\Subset Z$,
\[
\int_A |D\mathbf{v}|^p\,dx
\leq
\liminf_{\varepsilon\to0}
\int_A |P(Du_\varepsilon)|^p\,dx
\leq |A|,
\]
it follows that
\[
|D\mathbf{v}|\leq1
\qquad\text{a.e. on }Z.
\]
Therefore,
\[
H_{p/2}(D\mathbf{v})=0=\mathbf{w}
\qquad\text{a.e. on }Z.
\]
Combining the conclusions on $E$ and $Z$, we obtain
\[
w=H_{p/2}(D\mathbf{v})
\qquad\text{a.e. in }B_R,
\]
consequently,
\[
H_{p/2}(Du_\varepsilon)
\to H_{p/2}(D\mathbf{v})
\qquad\text{strongly in }L^2_{loc}(B_R).
\] 
Now passing to the limit as $\varepsilon\to 0$ in \eqref{approx3} and applying Fatou Lemma, we get
\begin{align}
\label{approx4}\nonumber
\int_{B_\rho}|DH_{p/2}(D\mathbf{v})|^2dx&\leq C||k(x)||_{L^s(B_{R+\delta})}^{\frac{p}{p-1}}\bigg(\int_{B_{R}}|u|^{p}+|Du|^pdx\bigg)^{\frac{q}{p-1}}\\ \nonumber
&+C||g(x)||_{L^\gamma(B_{R+\delta})}^{\frac{p}{p-1}}\bigg(\int_{B_{R}}|u|^{p}+|Du|^pdx\bigg)^{\frac{q}{p-1}}\\
&+C\int_{B_R}(1+|Du|^p)dx,
\end{align}
Our next aim is to prove that the function $\mathbf{v}\in W^{1,p}(\Omega)$ is a solution of \eqref{equazione nuova} with $u$ in place of $w$, i.e. in weak formulation, that the following identity
\begin{equation}
    \label{verifica2}
    -\int_{B_R}\bigg(a(x)(|D\mathbf{v}|-1)^{p-1}_+\frac{D\mathbf{v}}{|D\mathbf{v}|}\bigg)D\varphi \ dx=\int_{B_R}b(x,u)\varphi \ dx.
\end{equation}
holds $\forall\varphi\in W^{1,p}_0(B_R).$


In order to prove \eqref{verifica2}, using that $u_\varepsilon$ is a solution to problem \eqref{equ regolare}, we observe that  
\begin{align}\label{verifica3}
\nonumber
&-\int_{B_R}\langle a(x)H_{p-1}(D\mathbf{v}), D\varphi\rangle dx=-\int_{B_R}a(x)\langle H_{p-1}(D\mathbf{v}),D\varphi\rangle \ dx \\ \nonumber
&+\int_{B_R}\langle a_\varepsilon(x)H_{p-1}(Du_\varepsilon)+\varepsilon|Du_\varepsilon|^{p-2}Du_\varepsilon, D\varphi \rangle \ dx-\int_{B_R}\langle a_\varepsilon(x)H_{p-1}(Du_\varepsilon)+\varepsilon|Du_\varepsilon|^{p-2}Du_\varepsilon, D\varphi\rangle \ dx\\ \nonumber
&=-\int_{B_R}a(x)\langle H_{p-1}(D\mathbf{v}),D\varphi\rangle \ dx+\int_{B_R}\langle a_\varepsilon(x)H_{p-1}(Du_\varepsilon)+\varepsilon|Du_\varepsilon|^{p-2}Du_\varepsilon,D\varphi\rangle \ dx\\\nonumber
&+\int_{B_R}b_\varepsilon(x,u)\varphi \ dx+\int_{B_R}a(x)\langle H_{p-1}(Du_\varepsilon),D\varphi\rangle dx-\int_{B_R}a(x)\langle H_{p-1}(Du_\varepsilon),D\varphi\rangle dx\\\nonumber
&=-\int_{B_R}a(x)\langle H_{p-1}(D\mathbf{v})-H_{p-1}(Du_\varepsilon), D\varphi\rangle dx+\varepsilon\int_{B_R}\langle |Du_\varepsilon|^{p-2}Du_\varepsilon,D\varphi\rangle dx\\ \nonumber
&+\int_{B_R}\bigg(b_\varepsilon(x,u)-b(x,u)\bigg)\varphi \ dx+\int_{B_R}(a_\varepsilon(x)-a(x))\langle H_{p-1}(Du_\varepsilon), D\varphi\rangle dx+\int_{B_R}b(x,u)\varphi \ dx \\
&=:-J_\varepsilon+JJ_\varepsilon+JJJ_\varepsilon+JJJJ_\varepsilon+\int_{B_R}b(x,u)\varphi \ dx.
\end{align}
Our aim is to prove that $J_\varepsilon$, $JJ_\varepsilon$, $JJJ_\varepsilon$ and $JJJJ_\varepsilon$ go to $0$ as $\varepsilon$ go to $0.$ Indeed, we have
\begin{align}
  \label{j1} \nonumber
  |J_\varepsilon|&\leq M||D\varphi||_{L^p}\bigg(\int_{supp \ \varphi}|H_{p-1}(D\mathbf{v})-H_{p-1}(Du_\varepsilon)|^\frac{p}{p-1}dx\bigg)^{\frac{p-1}{p}}\\ \nonumber
  &\leq C||D\varphi||_{L^p}\bigg(\int_{supp \ \varphi}|H_{p/2}(D\mathbf{v})-H_{p/2}(Du_\varepsilon)|^{\frac{p}{p-1}}\bigg((|D\mathbf{v}|-1)_++(|Du_\varepsilon|-1)_+\bigg)^{\frac{p}{p-1}\frac{p-2}{2}}dx\bigg)^{\frac{p-1}{p}}\\ \nonumber
&\leq C||D\varphi||_{L^p}\bigg(\int_{supp \ \varphi}|H_{p/2}(D\mathbf{v})-H_{p/2}(Du_\varepsilon)|^2dx\bigg)^{\frac{1}{2}}\bigg(\int_{supp \ \varphi}\left((|D\mathbf{v}|-1)_++(|Du_\varepsilon|-1)_+\right)^pdx\bigg)^{\frac{p-2}{2p}}\\
&\leq C||D\varphi||_{L^p}\bigg(\int_{supp \ \varphi}|H_{p/2}(D\mathbf{v})-H_{p/2}(Du_\varepsilon)|^2dx\bigg)^{\frac{1}{2}}\bigg(\int_{B_R}|D\mathbf{v}|^p+|Du_\varepsilon|^pdx\bigg)^{\frac{p-2}{2p}},
\end{align}where we used assumption $(i)$ in \eqref{ipoa}, Lemma \ref{eq:BraAmb} with $\varepsilon=p-1$ and $\alpha=p/2$ and H\"older's inequality twice. Note that since $H_{p/2}(Du_\varepsilon) \to H_{p/2}(D\mathbf{v})$ strongly in $L^2_{loc}(B_R)$ and $\int_{B_R}|Du_\varepsilon|^pdx$ is bounded by \eqref{c}, we conclude that $J_\varepsilon \to 0$.\\ 
The estimate of $JJ_\varepsilon$ follows by using H\"older's inequality and estimate \eqref{c}
\begin{align}
    \label{j3} \nonumber
    |JJ_\varepsilon|&\leq \varepsilon||D\varphi||_{L^p}\bigg(\int_{supp \ \varphi} |Du_\varepsilon|^p dx\bigg)^{\frac{p-1}{p}}\\
    &\leq \varepsilon||D\varphi||_{L^p}\bigg[C\bigg(\int_{B_{R}} \!\!\!\!(1+|Du|^p) dx+|B_R|\bigg)^{\frac{p-1}{p}}\!\!+\!C\bigg(\int_{B_{R+\delta}} g(x)^\gamma dx \bigg)^{\frac{1}{\gamma}}\bigg(\int_{B_{R}}(|u|^p+|Du|^p) dx\bigg)^{\frac{q}{p}}\bigg],
\end{align} that goes to $0$ as $\varepsilon\to 0.$
To estimate $JJJ_\varepsilon$ we use the definition of $b_\varepsilon(x,u)$ and \eqref{5star}, as follows
\begin{align}
  \label{jj} \nonumber
 \int_{B_R}|b_\varepsilon(x,u)-b(x,u)|dx&=\int_{B_R}\bigg|\int_{B_1}(b(x+\varepsilon w,u)-b(x,u))\rho(w) dw\bigg|dx\\ \nonumber
 &\leq \int_{B_R}\int_{B_1}|b(x+\varepsilon w,u)-b(x,u)|\rho(w) dwdx\\ \nonumber
 &\leq |\varepsilon| \int_{B_R}\bigg(\int_{B_1}(k(x+\varepsilon w)+k(x))\rho(w)dw\bigg)|u|^q dx\\ \nonumber
 &\leq |\varepsilon| \int_{B_R}(k_\varepsilon(x)+k(x))|u|^q dx\\ \nonumber
 &\leq |\varepsilon| \bigg(\int_{B_R}(k_\varepsilon(x)+k(x))^{\frac{p^*}{p^*-q}}dx\bigg)^{\frac{p^*-q}{p^*}}\bigg(\int_{B_R}|u|^{p^*} dx\bigg)^{\frac{q}{p^*}}\\
 & \leq C|\varepsilon| \ ||k(x)||_{L^s(B_{R+\delta})}\ ||u||_{L^{p^*}(B_R)}^{\frac{q}{p^*}}
\end{align}
where we also used H\"older's inequality, \eqref{ke} and that $\frac{p^*}{p^*-q}<s,$ so we can conclude that $JJJ_\varepsilon \to 0$ as $\varepsilon \to 0$.
In order to estimate $|JJJJ_\varepsilon|$, we use H\"older's inequality, to obtain
\begin{align}
    |JJJJ_\varepsilon|\leq \bigg(\int_{B_R}(|a_\varepsilon(x)-a(x)||D\varphi|)^p dx\bigg)^{\frac{1}{p}}\bigg(\int_{B_R}|H_{p-1}(Du_\varepsilon)|^{\frac{p}{p-1}}dx\bigg)^{\frac{p-1}{p}},
    \end{align}
The sequence of functions $\vert{}(a_\varepsilon - a) D\varphi\vert{}^p$ converges to $0$ almost everywhere in $B_R$ and using assumptions $(ii)$ in \eqref{ipoa} and \eqref{ae}, we note that it is uniformly bounded by the integrable function $(2M)^p \vert{}D\varphi\vert{}^p \in L^1(B_R)$. 
Consequently, the Dominated Convergence Theorem ensures that$$(a_\varepsilon - a) D\varphi \to 0 \qquad \text{strongly in } L^p(B_R).$$ 
Combining this strong convergence with the uniform bound in $L^{p'}(B_R)$ for $\bigg(\int_{B_R}|H_{p-1}(Du_\varepsilon)|^{\frac{p}{p-1}}dx\bigg)^{\frac{p-1}{p}}$, we conclude that the integral converges to $0$ as $\varepsilon \to 0$.\\
To conclude the proof, we have to prove that $H_{p/2}(Du)=H_{p/2}(D\mathbf{v}).$ Since $u$ and $\mathbf{v}$ solve \eqref{equazione mia} and \eqref{equazione nuova} with $w=u$ respectively, we have
\begin{equation*}
\int_{B_R}a(x)\langle H_{p-1}(Du)-H_{p-1}(D\mathbf{v}), D\varphi\rangle dx=0
\end{equation*}
and, choosing as $\varphi=u-\mathbf{v}$, we get
\begin{align}
    \label{fineapprox}
    0\leq m\int_{B_R}|H_{p/2}(Du)-H_{p/2}(D\mathbf{v})|^2dx &\leq \int_{B_R}a(x)\langle H_{p-1}(Du)-H_{p-1}(D\mathbf{v}), Du-D\mathbf{v}\rangle dx=0,
\end{align}
where we used the ellipticity of $H_{p/2}(\xi)$ given by Lemma \ref{eq:BraAmb} and assumption \eqref{ipoa}. Hence \eqref{fineapprox} proves that $H_{p/2}(Du)=H_{p/2}(D\mathbf{v})$ and consequently, using this information in $\eqref{approx4}$ and that $R+\delta<R_0$, we conclude with the estimate

    \begin{align*}
&\int_{B_{R/2}}|D(H_{p/2}(Du))|^2dx\leq C\int_{B_R}(1+|Du|^p)dx\\
&+C\bigg(\int_{B_{R}}|u|^{p}+|Du|^pdx\bigg)^{\frac{q}{p-1}}\bigg[\bigg(\int_{B_{R_0}}|k(x)|^sdx\bigg)^{\frac{p}{s(p-1)}}+\bigg(\int_{B_{R_0}}|g(x)|^\gamma dx\bigg)^{\frac{p}{\gamma(p-1)}}\bigg].  
\end{align*}

\end{document}